\documentclass[]{amsart}

\usepackage{MnSymbol}
\usepackage[cmtip,arrow]{xy}  
\usepackage{pb-diagram,pb-xy}  
\usepackage{graphicx}    
\usepackage{amsfonts}
\usepackage[latin1]{inputenc}
\usepackage{caption}
\usepackage{subcaption}
\usepackage{comment}
\usepackage{enumerate}
\usepackage{amsmath}
\usepackage{bbm}
\usepackage{hyperref}
\usepackage{tkz-euclide}

\usepackage{tikz-cd}
\usepackage{tikz}
\usetikzlibrary{matrix}
\usetikzlibrary{positioning, calc}
\usetikzlibrary{shapes, fit}
\usetikzlibrary{arrows}  
\usetikzlibrary{calc,3d}
\usetikzlibrary{decorations,decorations.pathmorphing}
\usetikzlibrary{through}
\tikzset{ext/.style={circle, draw,inner sep=1pt},int/.style={circle,draw,fill,inner sep=1pt},nil/.style={inner sep=1pt}}
\tikzset{exte/.style={circle, draw,inner sep=1pt},inte/.style={circle,draw,fill,inner sep=3pt}}
\tikzset{diagram/.style={matrix of math nodes, row sep=1.5em, column sep=0.5em, text height=1.5ex, text depth=0.25ex}}
\tikzset{diagram2/.style={matrix of math nodes, row sep=0.5em, column sep=0.5em, text height=1.5ex, text depth=0.25ex}}

\newcommand{\aor}{\rcirclearrowleft}
\newcommand{\aol}{\lcirclearrowright}

\newcommand{\cP}{{\mathcal{P}}}
\newcommand{\cQ}{{\mathcal{Q}}}
\newcommand{\cM}{{{\mathcal M}}}
\newcommand{\Br}{{\mathsf{Br}}}
\newcommand{\CBr}{{\mathsf{CBr}}}

\newcommand{\CPT}{{\mathsf{CPT}}}
\newcommand{\BV}{{\mathsf{BV}}}
\newcommand{\BVKGra}{{\mathsf{BVKGra}}}
\newcommand{\BVKGraphs}{{\mathsf{BVKGraphs}}}
\newcommand{\BVGra}{{\mathsf{BVGra}}}
\newcommand{\BVGraphs}{{\mathsf{BVGraphs}}}
\newcommand{\Gra}{{\mathsf{Gra}}}
\newcommand{\Ger}{{\mathsf{Ger}}}
\newcommand{\Lie}{{\mathsf{Lie}}}
\newcommand{\hoLie}{{\mathsf{hoLie}_1}}
\newcommand{\FM}{{\mathsf{FM}}}
\newcommand{\FFM}{{\mathsf{FFM}}}

\newcommand{\Tpoly}{T_{\rm poly}}
\newcommand{\Dpoly}{D_{\rm poly}}
\newcommand{\TDpoly}{\tilde{D}_{\rm poly}} 
\newcommand{\Rformal}{\mathbb R^d_{\rm{formal}}}

\DeclareMathOperator{\vol}{vol}

\DeclareMathOperator{\id}{id}

\DeclareMathOperator{\End}{End}
\DeclareMathOperator{\Hom}{Hom}
\DeclareMathOperator{\sgn}{sgn}
\DeclareMathOperator{\F}{Forget_\infty} 

\newtheorem{theorem}{Theorem}[]
\newtheorem{corollary}[theorem]{Corollary}

\newtheorem{lemma}[theorem]{Lemma}
\newtheorem{proposition}[theorem]{Proposition}

\newtheorem{prop-def}[theorem]{Proposition/Definition}
\newtheorem{remark}[theorem]{Remark}
\newtheorem{convention}[theorem]{Convention}
\newtheorem{defi}{Definition}
\newtheorem{example}{Example}
 
\begin{document}

\title{BV Formality}

\author{Ricardo Campos}
\address{Institute of Mathematics\\ University of Zurich\\ Winterthurerstrasse 190 \\ 8057 Zurich, Switzerland}
\email{ricardo.campos@math.uzh.ch}

\begin{abstract}
We prove a stronger version of the Kontsevich Formality Theorem for orientable manifolds, relating the Batalin-Vilkovisky (BV) algebra of multivector fields and the homotopy BV algebra of multidifferential operators of the manifold.
\end{abstract}

\maketitle
\tableofcontents

\section{Introduction}

Given a manifold $M$,  the space of multidifferential operators of $M$, $\Dpoly(M)$ is a smooth version of the Hochschild complex of the functions on $M$. Both $\Dpoly(M)$ and the space $\Tpoly(M)$ of multivector fields of $M$  are (shifted) differential graded Lie algebras. These two objects are related by the Hochschild-Kostant-Rosenberg Theorem that provides us with a quasi-isomorphism $\Tpoly (M)\to \Dpoly(M)$. However, this map not compatible with the Lie structure.

 Searching for a canonical formal quantization of Poisson manifolds, in \cite{Kontsevich} M. Kontsevich establishes the existence of a homotopy Lie quasi-isomorphism $\Tpoly (M)\to \Dpoly(M)$ extending the Hochschild-Kostant-Rosenberg map. This map, nowadays called Kontsevich's Formality morphism, has a very explicit description involving integrals over configuration spaces of points when $M= \mathbb R^d$. 


Taking the wedge product into consideration $\Tpoly$ is a Gerstenhaber algebra, and even if $\Dpoly$ is not a Gerstenhaber algebra, its homology is in a standard way. It is natural to ask whether one can put a homotopy Gerstenhaber algebra structure on $\Dpoly$ that induces the usual Gerstenhaber algebra in the cohomology (Deligne's conjecture) and find a Formality morphism satisfying the Gerstenhaber structure up to homotopy. This question has been answered affirmatively by D. Tamarkin \cite{Tamarkin,Hinich}. 

In \cite{Note}, T. Willwacher uses a different model for the Gerstenhaber operad, the Braces operad, that acts naturally on $\Dpoly$ given the nature of the formulas. Willwacher proves in loc. cit. a homotopy Braces version of the Formality morphism.

In this paper we intend to take the final step on this chain of results by showing a BV version of the Formality Theorem(s).
As described in Section \ref{section:preliminaries}, we can endow both $\Tpoly(\mathbb R^d)$ and the cohomology of $\Dpoly(\mathbb R^d)$ with a degree $-1$ operator, extending the previous Gerstenhaber structures to BV algebra structures.

The cyclic structure of  $\Dpoly(\mathbb R^d)$ leads to the construction of $\CBr$, the Cyclic Braces operad which is a refinement of the Braces operad. We show that the operad $\CBr$ is quasi-isomorphic to $\BV$, the operad governing BV algebras, and the action of $\CBr$ on $\Dpoly(\mathbb R^d)$ descends to the canonical BV algebra structure on $\Dpoly(\mathbb R^d)$. In section \ref{section: main thm in Rd} we show that the BV action on $\Tpoly(\mathbb R^d)$ can be lifted to an action of $\CBr_\infty$, a resolution of $\CBr$ and we show the first main Theorem.

\begin{theorem}\label{main theorem}
There exists a $\CBr_\infty$ quasi-isomorphism $\Tpoly(\mathbb R^d) \to \Dpoly(\mathbb R^d)$.
\end{theorem}

The components of this morphism are defined through integrals similarly to Kontsevich's case.

The formality of the Cyclic Braces operad implies that in the previous Theorem $\CBr_\infty$ can be replaced any other cofibrant resolution of $\BV$, namely its minimal model or the Koszul resolution of $\BV$.\\

If we require orientability of the manifold $M$, the spaces $\Tpoly(M)$ and $H(\Dpoly(M))$ still have natural $BV$ structures. Using formal geometry techniques, together with the formalism of twisting of bimodules, in Section \ref{section:globalization} we show a global version Theorem \ref{main theorem}.

\begin{theorem}\label{main global}
Let $M$ be an oriented manifold. There exists a $\CBr_\infty$ quasi-isomorphism $\Tpoly(M) \to \Dpoly(M)$  extending Kontsevich's Formality morphism.
\end{theorem}

Applications of this theorem to string topology are expected.

\subsection{Acknowledgments}
I would like to thank Thomas Willwacher for proposing the problem and for the useful discussions and suggestions. I would also like to thank Ya\"el Fr\' egier and Bruno Vallette for related discussions. 

This work was partially supported by the Swiss National Science Foundation, grant 200021\_150012 and the NCCR SwissMAP funded by the Swiss National Science Foundation.

\subsection{Notation and conventions}
In this paper we work over the field $\mathbb R$ of real numbers, even though the ``algebraic" results hold in any field of characteristic zero.

All algebraic objects objects are differential graded, or dg for short, unless otherwise stated so.

 If $\cP$ is a 2-colored operad, we denote the space of operations with $m$ inputs in color $1$, $n$ inputs in color $2$ and output in color $i$ by $\cP^i(m,n)$ and we might denote $\cP$ by $\left(\cP^1,\cP^2\right)$.

\section{Preliminaries}\label{section:preliminaries}

\subsection{BV algebras}
Let us recall the definition of a BV algebra and also fix degree conventions.

\begin{defi}
A Batalin-Vilkovisky algebra or BV-algebra is a quadruplet $(A,\cdot,[\ , \ ], \Delta)$, such that:

\begin{itemize}
\item $(A, \cdot)$ is a (differential graded) commutative associative algebra,

\item $(A,[\ ,\ ])$ is a $1$-shifted Lie algebra (i.e., the bracket has degree -1),

\item $(A,\cdot,[\ ,\ ])$ is a Gerstenhaber algebra, i.e., for all $a\in A$ of degree $|a|$, the operator $[a,-]$ is a derivation of degree $|a|-1$.

\item $\Delta\colon A\to A$ is a unary linear operator of degree $-1$ such that $\Delta$ is a derivation of the bracket,

\item The bracket is the failure of $\Delta$ being a derivation for the product, i.e.,

$$[-,-]=\Delta \circ(-\cdot -) - (\Delta(-)\cdot - )-(-\cdot \Delta (-)).$$
\end{itemize}
\end{defi}

We denote by $\BV$, the operad governing BV algebras.

\subsection{Hochschild cochain complex}

In this section we recall the basics of Hochschild cohomology. For a more detailed introduction, along with the missing proofs, see \cite{cyclic homology}.

Let $A$ be a non-graded associative algebra.

For $f\colon A^{\otimes m} \to A$ and $g\colon A^{\otimes n} \to A$, we define $f\circ_i g \colon A^{\otimes {m+n-1}} \to A$, for $i=1,\dots, m$, to be the insertion of $g$ at the $i$-th slot of $f$, 
$$f\circ_i g (a_1,\dots, a_{m+n-1}) = f(a_1,\dots, a_{i-1}, g(a_i,\dots,a_{i+n-1}),\dots,a_{m+n-1}).$$

\begin{lemma}
Let $f\colon A^{\otimes m} \to A$ and $g\colon A^{\otimes n} \to A$. The operation $f\circ g  \colon A^{\otimes {m+n-1}} \to A)$ given by 
$$f\circ g = \sum_{i=1}^m (-1)^{i-1}f\circ_i g,$$ 
defines a pre-Lie product (of degree $-1$).
\end{lemma} 

This defines a $-1$ shifted graded Lie algebra structure on $\prod_{n\geq 0}\Hom(A^{\otimes n}, A)$. Let $\mu \colon A^{\otimes 2} \to A$ be the multiplication of the algebra.

 Since $A$ is an associative algebra, we have $$[\mu,\mu](a_1,a_2,a_3) = 2\mu(a_1,\mu(a_2,a_3)) - 2\mu(\mu(a_1,a_2),a_3) = 0,$$ i.e., $\mu$ is a Maurer-Cartan element of the Lie algebra $\prod_{n\geq 0}\Hom(A^{\otimes n}, A)$.

\begin{defi}
Let $A$ be an associative algebra. The Hochschild cochain complex of $A$, $(C^\bullet(A),d)$ is defined by
$$C^n(A) = \Hom(A^{\otimes n}, A); \quad \quad d = [\mu, \cdot].$$
\end{defi}

Explicitly, for $f\in C^n(A)$ and $a_i\in A$, the differential is given by $df(a_1,\dots,a_{n+1})=$ $$ = a_1f(a_2,\dots,a_{n+1}) + \sum_{i=1}^{n-1} (-1)^{i-1} f(a_1,\dots, a_ia_{i+1},\dots, a_n) + (-1)^nf(a_1,\dots,a_{n})a_{n+1}.$$

\begin{defi}
The Hochschild cohomology of an associative algebra $A$ is the cohomology of the complex $C^\bullet(A)$ and is denoted by $HH^\bullet (A)$.
\end{defi}

\begin{defi}
Let $f\in C^m(A)$ and $g\in C^n(A)$. The cup product on Hochshild cochains $f\cup g\in C^{m+n}(A)$ is defined by
$$f\cup g (a_1\dots,a_{m+n}) = f(a_1,\dots,a_m)\cdot g(a_{m+1},\dots,a_{n+m}),$$
\end{defi}

The cup product is trivially associative but, in general, non-commutative and it does not satisfy the desired compatibility with the Lie bracket. However, as the following proposition tells us, this is rectified at the cohomological level.

\begin{proposition}
The cup product and the Lie bracket above defined, induce a Gerstenhaber algebra structure on $HH^\bullet (A)$.
\end{proposition}

\subsection{Multidifferential operators}\label{Dpoly as hochschild}
Let $M$ be an oriented manifold. One of the central objects of this paper is the space of multidifferential operators of $M$, which are a smooth analog of the Hochschild cochain complex.

\begin{defi}

Let $A= C^\infty(M)$, the algebra of smooth functions of $M$. The space of multidifferential operators $\Dpoly^\bullet (M)$ or just $\Dpoly$ if there is no ambiguity, is a subcomplex of $C^\bullet (A)$, given by

$$ \Dpoly^n  =  \left\lbrace D\colon C^\infty(M)^{\otimes n} \to C^\infty(M) \Big|  D \stackrel{\text{locally}}{=} \sum f\frac{\partial}{\partial x_{I_1}}\otimes \dots\otimes \frac{\partial}{\partial x_{I_n}}\right\rbrace,$$
where the $I_j$ are finite sequences of indices between $1$ and $dim(M)$ and $\frac{\partial}{\partial x_{I_j}}$ is the multi-index notation representing the composition of partial derivatives.
\end{defi}

We will now describe an action of the group $C_{n+1} = \langle\sigma_n | \sigma_n^{n+1} = e \rangle$ on $\Dpoly^n$.

Since every multidifferential operator is uniquely determined by evaluation on the compactly supported functions $C^\infty_c (M)$, then, $\Dpoly^n$, for $n\geq 1$ can be seen as a subspace of $\Hom(C^\infty_c (M)^{\otimes n} , C^\infty_c (M)$. One can equally see $\Dpoly$ as a subspace of $\Hom(C^\infty_c (M)^{\otimes n+1} , \mathbb R)$ in the following way: 

Let us denote by $\vol$ the given volume form $M$. We identify $D \in \Dpoly^n \subset \Hom(C^\infty_c (M)^{\otimes n} , C^\infty_c (M))$, with 
$$\left[f_1\otimes \dots \otimes f_{n+1} \mapsto \int_{M} f_1D(f_2,\dots,f_{n+1}) \vol\right] \in \Hom(C^\infty_c (M)^{\otimes n+1} , \mathbb R).$$

The reverse identification can be obtained by integrating by parts in order to remove differential operators from $f_1$.

 From now on we drop the $M$ as the domain of integration and the $\vol$ to make the notation lighter.

There is an action of $C_{n+1}$ on $\Dpoly^n\subset \Hom(C^\infty_c (M)^{\otimes n+1} , \mathbb R)$ is given by the cyclic permutation of the inputs. 

$$\int f_1D(f_2,\dots,f_{n+1}) = \int f_2 D^\sigma(f_3\dots,f_{n+1},f_1).$$

\begin{defi}
The Connes $B$ operator on $\Dpoly$, is the map $B\colon \Dpoly^\bullet \to \Dpoly^{\bullet-1}$ defined for all $D\in \Dpoly^n$ by

$$B(D)(f_1,\dots,f_{n-1}) =\sum_{k=0}^n (-1)^k D^{\sigma^k}(1,f_1,\dots,f_{n-1}),\quad  \forall f_i\in C^\infty(M).$$
\end{defi}

\begin{proposition}
The $B$ operator induces a well defined map in the cohomology of $\Dpoly$ and together with the Lie bracket and cup product defined in the previous section induces a $BV$-algebra structure in $H^\bullet(\Dpoly)$.
\end{proposition}

The proposition can be proved ``by hand", but also will also follow from the result that the operad $\CBr$, whose homology is the $\BV$ operad, acts on $\Dpoly$.
\subsection{Multivector fields}
\begin{defi}
Let $M$ be an oriented manifold. The graded vector space $\Tpoly(M)$ or just $\Tpoly$ of multivector fields on $M$ is 
$$\Tpoly^\bullet = \Gamma(M, {\bigwedge}^\bullet T_M),$$
where $T_M$ is the tangent bundle of $M$.
\end{defi}

$\Tpoly$ has a natural Gerstenhaber algebra structure by taking as product the wedge product of multivector fields and as bracket, the Schouten-Nijenhuis bracket, i.e., the unique $\mathbb R$-linear bracket satisfying
$$[X,Y\wedge Z] = [X,Y]\wedge Z + (-1)^{(|X|-1)(|Y|-1)}Y\wedge [X,Z],\ \forall X,Y,Z\in \Tpoly^\bullet$$
that restricts to the usual Lie bracket of vector fields.

We can define a map $f\colon \Tpoly^\bullet(M) \to \Omega^{n-\bullet}_{dR}(M)$ that sends a multivector field to its contraction with the volume form of $M$.

This map is easily checked to be an isomorphism of vector spaces. We define the divergence operator div to be the pullback of the de Rham differential via $f$, i.e. div$=f^{-1}\circ d_{dR}\circ f$.

A series of straightforward calculations prove the following:

\begin{proposition}
The space $\Tpoly^\bullet (M)$, with the wedge product, the  Schouten-Nijenhuis bracket and the divergence operator forms a $BV$-algebra.
\end{proposition}

\section{Cyclic Swiss Cheese type operads}

\subsection{Cyclic operads}
The standard notion of an operad is used in order to describe operations on a certain vector space with a given number of inputs and one output. A symmetric operad is used when one wants to take into consideration the symmetries on the inputs. The notion of a \textit{cyclic operad} \cite{Cyclic Operads, Loday-Vallette}, introduced by Getzler and Kapranov, arises when one considers the output as an \textit{additional input} that can be cyclically permuted along with the remaining inputs.
This can arise naturally in many situations, for example, when one is given a finite dimensional vector space $V$ equipped with a non-degenerate symmetric bilinear form, the space $\Hom(V^{\otimes n},V)$ can be identified with $\Hom(V^{\otimes n+1}, \mathbb R)$.

\begin{defi}
A cyclic operad on a symmetric monoidal category $(\mathcal C,\otimes, I, s)$ is the data of a non-symmetric operad $\cP$ and a right action of $C_{n+1}= \langle\sigma_n | \sigma_n^{n+1} = e \rangle$, the symmetric group of order $n+1$ on $\cP(n)$ satisfying the following axioms:

\begin{enumerate}[a)]

\item  The cyclic action on the unit in $\cP(1)$ is trivial.

\item For every $m,n\geq 1$, the diagram

\begin{tikzcd}
\cP(m)\otimes\cP(n) \arrow{r}{\circ_1} \arrow{d}{\sigma_m\otimes \sigma_n} & \cP(m+n-1)\arrow{dd}{\sigma_{m+n-1}}\\
\cP(m)\otimes \cP(n) \arrow{d}{s} &   \\
\cP(n) \otimes \cP(m) \arrow{r}{\circ_n} & \cP(m+n-1)\\
\end{tikzcd}

commutes.

\item For every $m,n\geq 1$ and $2\leq i \leq m$, the following diagram commutes:

\begin{tikzcd}
\cP(m)\otimes\cP(n) \arrow{r}{\circ_i} \arrow{d}{\sigma_m\otimes \id_{\cP(n)}} & \cP(m+n-1)\arrow{d}{\sigma_{m+n-1}}\\
\cP(m) \otimes \cP(n) \arrow{r}{\circ_{i-1}} & \cP(m+n-1).\\
\end{tikzcd}

\end{enumerate}
\end{defi} 

\subsection{Operad of Cyclic Swiss Cheese type}

\begin{defi}
Let $\cP$ be a 2-colored operad that is non-symmetric in color 2. We say that $\cP$ is of  Swiss Cheese type if $\cP^1(m,n) = 0$ if $n>0$.

Furthermore, $\cP$ is said to be of Cyclic Swiss Cheese type if these two following additional conditions hold:\begin{itemize}
\item The cyclic group of order $n+1$, $C_{n+1}$ acts on the right on $\cP^2(m,n)$ satisfying the same axioms as the axioms of a cyclic operad,
\item The cyclic action is $\cP^1$ equivariant,
\item There is a distinguished element $\mathbbm 1_{\cP}\in \cP^2(0,0)$.
\end{itemize}
\end{defi}

 For simplicity of notation we denote $\cP^1(m,0)$ by $\cP^1(m)$. Using the distinguished element $\mathbbm 1_{\cP}$ we define the ``forgetful" map $\F \colon \cP^2(m,n) \to \cP^2(m,n-1)$
 by $\F(p) = p^{\sigma_n}(\id_{\cP^1},\dots, \id_{\cP^1};\id_{\cP^2},\dots,\id_{\cP^2},\mathbbm 1_{\cP})$.

A morphism $\cP\to\cQ$ of Cyclic Swiss Cheese type operads is a colored operad morphism that is equivariant with respect to the cyclic action and sends $\mathbbm 1_{\cP}$ to $\mathbbm 1_{\cQ}$.

\subsection{Examples}

\subsubsection{Multidifferential operators as an operad}\label{Dpoly}

Let $M$ be an oriented manifold. The operad of multidifferential operators  $\TDpoly(M)$, or just $\TDpoly$, is a differential graded operad concentrated in degree zero with zero differential given by
$$ \TDpoly^n := \TDpoly(n)  =  \left\lbrace D\colon C^\infty(M)^{\otimes n} \to C^\infty(M) \Big|  D\stackrel{\text{loc.}}{=} \sum f\frac{\partial}{\partial x_{I_1}}\otimes \dots\otimes \frac{\partial}{\partial x_{I_n}}\right\rbrace.\footnote{This is almost the object introduced in Section \ref{Dpoly as hochschild}. The tilde is a reminder that there is no grading or differential.}$$

The operadic structure is the one induced by the endomorphisms operad of $C^\infty(M)$, i.e., given by composition of operators. 
 As any other operad, $\TDpoly$ can be seen as a 2-colored operad simply by declaring that there are no operations with inputs or outputs in color 1.   To endow $\TDpoly$ with a Cyclic Swiss Cheese Operad type structure we use the cyclic action defined in Section \ref{Dpoly as hochschild} and the distinguished element $\mathbbm 1 \in \TDpoly^0 = C^\infty(M)$ is defined to be the constant function $1$. 

For every $D \in \TDpoly^n \subset \Hom(C^\infty_c (\mathbb R^d)^{\otimes n} , C^\infty_c (\mathbb R^d))$ we have $$\F(D)= \int D(\cdot) \vol \in \Hom(C^\infty_c (\mathbb R^d)^{\otimes n} , \mathbb R). $$

\subsubsection{Configurations of framed points}\label{FFM}
The Fulton-MacPherson topological operad $\FM_2$, introduced by Getzler and Jones \cite{Getzler-Jones} is constructed in such a way that the $n$-ary space $\FM_2(n)$ is a compactification of the configuration space of points labeled $1,\dots,n$ in $\mathbb R^2$, modulo scaling and translation. The spaces $\FM_2(n)$ are manifolds with corners with each boundary stratum representing a set of points that got infinitely close.

The first few terms are
\begin{itemize}
	\item $\FM_2(0) = \emptyset$,
	\item $\FM_2(1) = \{*\}$,
	\item $\FM_2(2) = S^1$.
\end{itemize}
The operadic composition $\circ_i$ is given by inserting a configuration at the boundary stratum at the point labeled by $i$.
 For details on this construction see also  \cite[Part IV]{FM} or \cite{Kontsevich}.

\begin{defi}
Let $\cP$ be a topological operad such that there is an action of topological group $G$ on every space $\cP(n)$ and the operadic compositions are $G$-equivariant. The semi-direct product $G\ltimes \cP$ is a topological operad with $n$-spaces

 $$(G\ltimes \cP)(n) = G^n \times \cP(n),$$ 
and composition given by $$(\overline{g},p) \circ_i (\overline{g'},p') = \left(g_1,\dots, g_{i-1}, g_ig_1',\dots,g_ig_m',g_{i+1}, \dots,g_n , p\circ_i (g_i \cdot p')\right),$$ where $\overline g = (g_1,\dots,g_n)$ and $\overline {g'}=(g'_1,\dots,g'_m) $.

\end{defi}

The topological group $S^1$ acts on $\FM$ by rotations.
We define the Framed Fulton-MacPherson topological operad $\FFM_2$ to be the semi-direct product $S_1 \ltimes \FM_2$. Equivalently, $\FFM_2(n)$ is the compactification of the configuration space of points modulo scaling and translation such that at every point we assign a frame, i.e., an element of $S^1$. When the operadic composition is performed, the configuration inserted rotates according to the frame on the point of insertion.

We denote by $\mathbb H_{m,n}$, the space of configurations of $m$ points in the upper half plane labeled by $1,\dots,m$ and $n$ points at the boundary, labeled by $\overline 1,\dots, \overline n$, modulo scaling and horizontal translations, with a similar compactification. Similarly, $\mathsf{F}\mathbb H_{m,n}$ shall be the compactification of the space of configurations of $m$ framed points in the upper half plane and $n$ non-framed points at the boundary. These spaces are considered unital in the sense that $\mathsf{F}\mathbb H_{0,0}$ is topologically a point, instead of the empty space.

Together they form a Swiss Cheese type topological operad $\cP$, with $\cP^1 = \FFM_2$ and 
$\cP^2= \mathsf F\mathbb H$ with composition of color $2$ being insertion of the corespondent configuration in the boundary stratum and composition of color $1$ on the vertex labeled by $i$ being the insertion at the boundary stratum at the point $i$ after applying the corresponding rotation given by the frame of $i$. We shall consider that a framing pointing upwards represents the identity of $S^1$, see Figure \ref{composition ffm}.

\begin{figure}

\begin{tikzpicture}[scale=1]

\node[label=left:$1$][fill,circle,scale=0.25] (1) at (1.8,1) {};
\node (v1) at (1,0) {};
\node (v2) at (3,0) {};

\draw [-] (v1)--(v2);
\draw[thick] (1) --  +(-45:0.25cm);
\draw [dotted] (1)-- +(90:1cm);
\draw (1)+(0,0.2) arc (90:-45:0.18cm);

\node  at (2.1,1.3)  {$\psi$};

\end{tikzpicture}
 $\circ_1$
  \begin{tikzpicture}[scale=1]

\node[label=left:$1$][fill,circle,scale=0.25] (1) at (1.8,1)  {};
\path (1) -- +(-35:0.5cm) node[label=below:$2$][fill,circle,scale=0.25] (2) {};

\draw[thick] (1) --  +(130:0.25cm);
\draw[thick] (2) --  +(10:0.25cm);
\draw [dotted] (1)-- +(90:1cm);
\draw (1)+(0,0.2) arc (90:-35:0.18cm);
\draw [dotted] (1)--(2);

\node  at (2.1,1.3)  {$\varphi$};
\end{tikzpicture}
=
\begin{tikzpicture}[scale=1]

\node[label=above left:$\small{1}$][fill,circle,scale=0.25] (1) at (1.8,1)  {};
\draw [dotted] (1)--+(-170:0.5cm) node[label=left:$\small{2}$][fill,circle,scale=0.25] (2)  {};

\node (v1) at (1,0) {};
\node (v2) at (3,0) {};

\draw [-] (v1)--(v2);

\draw [dotted] (1)-- +(90:1cm);
\draw (1)+(0,0.2) arc (90:-170:0.18cm);

\draw[thick] (1) --  +(-5:0.25cm);
\draw[thick] (2) --  +(-125:0.25cm);

\node  at (2.3,0.75)  {\footnotesize{$\psi+\varphi$}};

\end{tikzpicture}
\caption{Composition of an element of $\FFM_2$ with an element in $\mathsf{F}\mathbb H$.}\label{composition ffm}
\end{figure}

In fact they can be endowed with a Cyclic Swiss Cheese type operad structure. 

The open upper half plane is isomorphic to the Poicar\'e disk via a conformal (angle preserving) map. This isomorphism sends the boundary of the plane to the boundary of the disk except one point, that we label by $\infty$. We define the cyclic action of $C_{n+1}$ in $\FFM_2(m,n)$ by cyclic permutation of the point labeled by infinity with the other points at the boundary.

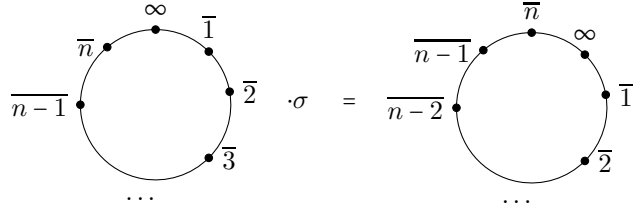
\begin{figure}[h]
\begin{tikzpicture}

\node (pic1) at (0,0)  
{\begin{tikzpicture}
\draw (0,0) arc (90:45:1cm) node [label= $\overline 1$] (1)[int]{} 
arc (45:10:1cm) node [label=right:$\overline 2$] (2)[int]{} 
arc (10:-45:1cm) node [label=right:$\overline 3$] (3)[int]{} 
arc (-45:-100:1cm) node [label=below:$\dots$]{} 
arc (-100:-180:1cm) node [label=left:$\overline{n-1}$] [int]{} 
arc (-180:-230:1cm) node [label=left:$\overline n$][int]{} 
arc (-230:-270:1cm) node [label=$\infty$][int]{} ;
\end{tikzpicture}};

\node (text) at (2.5,0) {$\cdot \sigma\quad =$};

\node (pic2) at (5,0) 
{\begin{tikzpicture}
\draw (0,0) arc (90:45:1cm) node [label= $\infty$] [int]{} 
arc (45:10:1cm) node [label=right:$\overline 1$] [int]{} 
arc (10:-45:1cm) node [label=right:$\overline 2$] (3)[int]{} 
arc (-45:-100:1cm) node [label=below:$\dots$]{} 
arc (-100:-180:1cm) node [label=left:$\overline{n-2}$] [int]{} 
arc (-180:-230:1cm) node [label=left:$\overline {n-1}$][int]{} 
arc (-230:-270:1cm) node [label=above:$\overline n$][int]{};
\end{tikzpicture}};
\end{tikzpicture}
\end{figure}

The element $\mathbbm 1$ is defined to be the unique point in $\mathsf F\mathbb H_{0,0}$. Insertion of this element represents forgetting a certain point at the boundary.

The forgetful map is defined by forgetting the point at infinity and labeling the first point as the new $\infty$ point and the previous $\overline n$ becomes the new $\overline{n-1}$.

$$
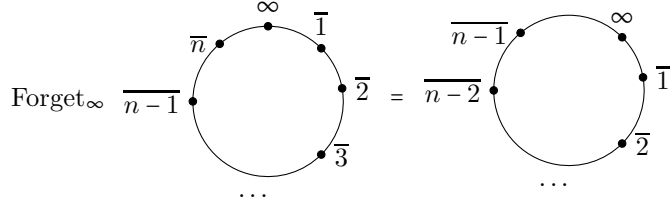
\begin{tikzpicture}
\node at (-2.5,0) {$\F$};

\node at (0,0) {
\begin{tikzpicture}
\draw (0,0) arc (90:45:1cm) node [label= $\overline 1$] (1)[int]{} 
arc (45:10:1cm) node [label=right:$\overline 2$] (2)[int]{} 
arc (10:-45:1cm) node [label=right:$\overline 3$] (3)[int]{} 
arc (-45:-100:1cm) node [label=below:$\dots$]{} 
arc (-100:-180:1cm) node [label=left:$\overline{n-1}$] [int]{} 
arc (-180:-230:1cm) node [label=left:$\overline n$][int]{} 
arc (-230:-270:1cm) node [label=$\infty$][int]{} ;
\end{tikzpicture}};

\node at (2,0) {$=$};

\node at (4,0) {
\begin{tikzpicture}
\draw (0,0) arc (90:45:1cm) node [label= $\infty$] [int]{} 
arc (45:10:1cm) node [label=right:$\overline 1$] [int]{} 
arc (10:-45:1cm) node [label=right:$\overline 2$] (3)[int]{} 
arc (-45:-100:1cm) node [label=below:$\dots$]{} 
arc (-100:-180:1cm) node [label=left:$\overline{n-2}$] [int]{} 
arc (-180:-230:1cm) node [label=left:$\overline {n-1}$][int]{} 
arc (-230:-270:1cm) node {} ;
\end{tikzpicture}};
\end{tikzpicture}
$$
\subsubsection{Two kinds of graphs}\label{BVKGra}

A directed graph $\Gamma$ is the data of a finite set of vertices, $V(\Gamma)$ and a set of edges $E(\Gamma)$ consisting of ordered pairs of vertices, that is, a subset of $V(\Gamma)\times V(\Gamma)$. Notice that tadpoles (edges connecting a vertex to itself) are allowed.\\

Let $\BVKGra'(m,n)$ be the graded vector space spanned by directed graphs with $m$ vertices of \textit{type I} labeled with the numbers $\{1,\dots, m\}$, $n$ labeled with the numbers $\{ \overline 1,\dots, \overline n\}$ of \textit{type II} and edges labeled with the numbers $\{1,\dots,\#\text{edges}\}$, such that there are no edges starting on a vertex of type II. The degree of a graph is $-\#\text{edges}$, i.e., every edge has degree -1.
For every non-negative integer $d$, there is an action of $\mathbb S_{d}$ on $\CPT'_{-d}(n)$ by permutation of the labels of the edges.

We define the space $\BVKGra$ of BV Kontsevich Graphs by 
$$\BVKGra (m,n) := \bigoplus_d \BVKGra'_{-d}(m,n) \otimes_{\mathbb S_{d}} \sgn_d, $$
where $\sgn_d$ is the sign representation.

We define the space of BV Graphs, $\BVGra(n):= \BVKGra(n,0)$. There is a natural $\mathbb S_n$ action by permutation of the labels and we define a symmetric operad structure in $\BVGra$ by setting the composition $\Gamma_1 \circ_i \Gamma_2$ to be the insertion of $\Gamma_2$ in the $i$-th vertex of $\Gamma_1$ and sum over all possible ways of connecting the edges incident to $i$ to $\Gamma_2$.

We can form a Swiss Cheese type operad by setting $\BVGra$ to be the operations in color $1$ and $\BVKGra$ to be the operations in color $2$, considering the symmetric action permuting the labels of type I vertices and ignoring the symmetric action of type II vertices. The partial compositions are given as in $\BVGra$, i.e., by insertion on the corresponding vertex and connecting in all possible ways. 

The type II vertices in $\BVKGra$ will be later seen as \textit{boundary vertices} when we relate $\BVKGra$ with $\mathsf F \mathbb H$, and since we wish to distinguish between $\BVGra(\cdot)$ and $\BVKGra(\cdot,0)$, we draw the latter with a line passing by the type II vertices.

 
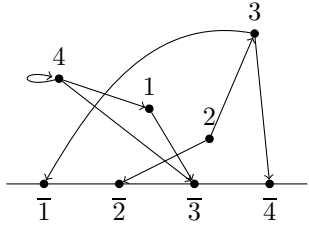
\begin{tikzpicture}[scale=1]

\node[label=$1$] (1) at (1.4,1) [int] {};
\node[label=$2$] (2) at (2.2,0.6) [int] {};
\node[label=$3$] (3) at (2.8,2) [int] {};
\node[label=$4$] (4) at (0.2,1.4) [int] {};
\node[label=below:$\overline 1$] (b1) at (0,0) [int] {};
\node[label=below:$\overline 2$] (b2) at (1,0) [int] {};
\node[label=below:$\overline 3$] (b3) at (2,0) [int] {};
\node[label=below:$\overline 4$] (b4) at (3,0) [int] {};

\draw [->] (1)--(b3);
\draw [->] (2)--(b2);
\draw [->] (2)--(3);
\draw [->] (4)--(b3);
\draw [->] (4) --  (1);
\draw [->] (3)  to [out=170,in=60] (b1);
\draw [->] (3)--(b4);

\path[->] (4) edge  [loop left] ();

\draw  (-0.5,0)--(3.5,0);
\end{tikzpicture}

The space $\BVKGra(m,n)$ forms a graded commutative  algebra with product of two graphs defined by superposing the vertices and taking the union of the edges. This algebra generated by 

$\Gamma^i_j:=$
 \begin{tikzpicture}[scale=1]

\node[label=$i$] (i) at (1.4,1) [int] {};
\node[label=below:$\overline 1$] (v1) at (1,0) [int] {};
\node at (1.5,-0.4) {\dots};
\node[label=below:$\overline j$] (j) at (2,0) [int] {};
\node at (2.5,-0.4) {\dots};
\node[label=below:$\overline n$] (n) at (3,0) [int] {};

\draw [->] (i)--(j);
\draw  (0.5,0)--(3.5,0);
\end{tikzpicture},with $1\leq i\leq m$ and $1\leq j \leq n$ 
and 

$\Gamma^{i,j}:=$ 
 \begin{tikzpicture}[scale=1]

\node[label=$i$] (i) at (1.4,1) [int] {};
\node[label=below:$\overline 1$] (v1) at (1,0) [int] {};
\node (dots) at (2,-0.4) {\dots};
\node[label=below:$\overline n$] (n) at (3,0) [int] {};
\node[label=$j$] (j) at (2,1.3) [int] {};

\draw [->] (i)--(j);
\draw  (0.5,0)--(3.5,0);
\end{tikzpicture},with $1\leq i,j\leq m$. For simplicity, the dependance of $m$ and $n$ is dropped from the notation.\\

Let $\Gamma^i_j \in \BVKGra(m,n)$. The action of the generator $\sigma$ of $C_{n+1}$ on $\Gamma^i_j$ is $\sigma(\Gamma^i_j)= \Gamma^i_{j-1}$ if $j\ne 1$ and 
 $\sigma(\Gamma^i_1) = -\sum_{k=1}^n \Gamma^i_k - \sum_{k=1}^m \Gamma^{i,k}$.
 The action of $\sigma$ on $\Gamma^{i,j}\in\BVKGra(m,n)$ is $\sigma(\Gamma^{i,j})= \Gamma^{i,j}$, for $1\leq i,j\leq m$.

We define the cyclic action of $C_{n+1} = \langle\sigma | \sigma^{n+1}=e\rangle$ on one-edge graphs of $\BVKGra(m,n)$ by $\sigma(\Gamma^i_j)= \Gamma^i_{j-1}$ if $j\ne 1$ and 
 $\sigma(\Gamma^i_1) = -\sum_{k=1}^n \Gamma^i_k - \sum_{k=1}^m \Gamma^{i,k}$.
 The action of $\sigma$ on $\Gamma^{i,j}\in\BVKGra(m,n)$ is defined by $\sigma(\Gamma^{i,j})= \Gamma^{i,j}$, for $1\leq i,j\leq m$.
 
Since $\sigma^2(\Gamma^i_1) = \Gamma^i_m$, we have that $\sigma^{n+1}$ acts as the identity in every one-edge graph, therefore the action of $C_{n+1}$ is well defined.

 We extend this action to $\BVKGra(m,n)$ by declaring that the action distributes over a product of graphs (i.e., making the cyclic action a morphism of unital algebras).
 
 The element $\mathbbm 1\in \BVKGra(0,0)$ is the empty graph, the unique graph with no vertices. The insertion $\Gamma \circ_j \mathbbm 1$ is zero if there is any edge incident to the vertex labeled by $\overline j$ or, if there are no such edges, it forgets the vertex labeled by $j$.

\section{The Cyclic Braces operad}

In \cite{KS}, Kontsevich and Soibelman introduced an operad that they call minimal operad that acts naturally on the Hochschild cochain complex of $A_\infty$ algebras. They show that this operad is quasi-isomorphic to $\Ger$, the operad governing Gerstenhaber algebras (see also \cite{Homology of Braces}). In this paper we call this operad $\Br$, standing for Braces. In this section we introduce the Cyclic Braces operad, which is a refinement of the Braces operad that is meant to take into account the a unit and a cyclic action. A similar operad was constructed by Ward in \cite{Alternative paper}.

\subsection{The Cyclic Planar Trees operad}
Let $\CPT'(n)$ be the graded vector space spanned by rooted planar trees with vertices labeled with the numbers in $\{1,\dots, n\}$ with the additional feature that every vertex can have additional edges connecting to a symbol $\mathbbm 1$ and every vertex has a marked edge, that can be one of the additional edges\footnote{In fact, we want at most one edge connecting to $\mathbbm 1$ per vertex and for vertices having an edge connecting to $\mathbbm 1$ we want to force the marked edge to be that one, but imposing this condition directly would not be stable by the composition that we define below. This is resolved by considering a quotient of $\CPT'$, rather than a subspace.}. The non-root edges (including the ones connecting to $\mathbbm 1$) are labeled by the numbers $\{1,\dots, \#\text{edges}\}$. The degree of a rooted planar tree is $-\#\text{edges}$. For every non-negative integer $d$, there is an action of $\mathbb S_{d}$ on $\CPT'_{-d}(n)$ by permuting the labels of the edges.

We define the operad $\CPT$ of Cyclic Planar Trees by 
$$\CPT (n) := \bigoplus_d \CPT''_{-d}(n) \otimes_{\mathbb S_{d}} \sgn_d, $$
where  $\sgn_d$ is the sign representation and $\CPT''$ is the quotient of $\CPT'$ by trees in which there is a vertex is connected to an element $\mathbbm 1$ whose mark is not pointing towards $\mathbbm 1$.

The operadic composition $T_1 \circ_j T_2$ is given by inserting the tree $T_2$ in the vertex labeled $j$ of the tree $T_1$, orienting the root of $T_2$ with the marking at the vertex $j$ of $T_1$, forgetting both the root and the mark at the vertex $j$ and reconnecting all incident edges in all planar possible ways. 

Since it it unambiguous, for simplicity of the drawing we draw only a mark between two edges when some vertex is connected to $\mathbbm 1$.

\begin{example}
Examples of insertion:

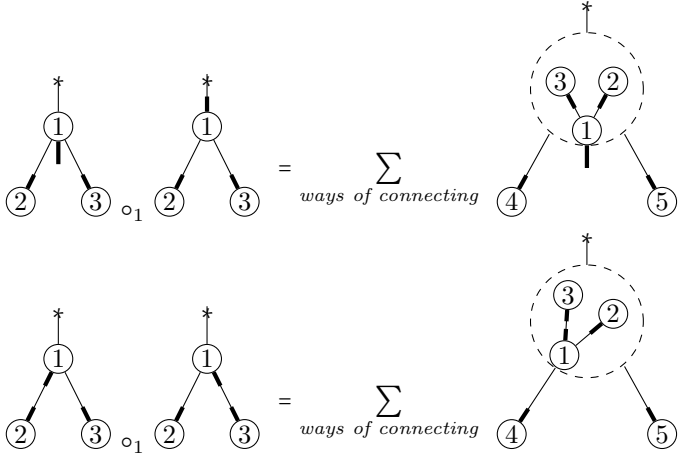
\begin{tikzpicture}[scale=1]

\node (ro) at (0,0.6) {$*$};
\node (e0) at (0,0) [ext] {$1$};
\draw[ultra thick] (e0) -- +(-90:0.5cm);
\node (e1) at (-0.5,-1) [ext] {$2$};
\node (e2) at (0.5,-1) [ext] {$3$};

\draw (e0) -- +(0,0.6);
\draw (e1) -- (e0);
\draw (e2) -- (e0);

\draw[ultra thick] (e2) --  ($(e2)!0.4cm!(e0)$);
\draw[ultra thick] (e1) --  ($(e1)!0.4cm!(e0)$);
\end{tikzpicture}
$\circ_1$
\begin{tikzpicture}[scale=1]

\node (ro) at (0,0.6) {$*$};
\node (e0) at (0,0) [ext] {$1$};

\node (e1) at (-0.5,-1) [ext] {$2$};
\node (e2) at (0.5,-1) [ext] {$3$};

\draw (e0) -- +(0,0.7);
\draw (e1) -- (e0);
\draw (e2) -- (e0);
\draw[ultra thick] (e0) --  ($(e0)!0.4cm!(ro)$);
\draw[ultra thick] (e2) --  ($(e2)!0.4cm!(e0)$);
\draw[ultra thick] (e1) --  ($(e1)!0.4cm!(e0)$);
\end{tikzpicture}
\begin{tikzpicture}
\node at (0,0) {$=\displaystyle\sum_{\text{ways of connecting}}$};
\end{tikzpicture}
\begin{tikzpicture}

\node (ro) at (0,0.6) {$*$};

\node (e4) at (-1,-2) [ext] {$4$};
\node (e5) at (1,-2) [ext] {$5$};

\draw [dashed] (0,-0.5) circle (0.75);

\node (e1) at (0,-1.05) [ext] {$1$};

\node (e2) at (-0.35,-0.4) [ext] {$3$};
\node (e3) at (0.35,-0.4) [ext] {$2$};

\draw (0,0.25) -- (0,0.6);
\draw (e1) -- (e2);
\draw (e1) -- (e3);
\draw[ultra thick] (e1) --  +(0,-0.5);

\draw (e5)  to  (0.5,-1.1);
\draw (e4)  to  (-0.5,-1.1);

\draw[ultra thick] (e5) --  ($(e5)!0.4cm!(0.5,-1.1)$);
\draw[ultra thick] (e4) --  ($(e4)!0.4cm!(-0.5,-1.1)$);
\draw[ultra thick] (e2) --  ($(e2)!0.4cm!(e1)$);
\draw[ultra thick] (e3) --  ($(e3)!0.4cm!(e1)$);

\end{tikzpicture}

\begin{tikzpicture}[scale=1]

\node (ro) at (0,0.6) {$*$};
\node (e0) at (0,0) [ext] {$1$};
\node (e1) at (-0.5,-1) [ext] {$2$};
\node (e2) at (0.5,-1) [ext] {$3$};

\draw (e0) -- +(0,0.6);
\draw (e1) -- (e0);
\draw (e2) -- (e0);

\draw[ultra thick] (e2) --  ($(e2)!0.4cm!(e0)$);
\draw[ultra thick] (e1) --  ($(e1)!0.4cm!(e0)$);
\draw[ultra thick] (e0) --  ($(e0)!0.4cm!(e1)$);

\end{tikzpicture}
$\circ_1$
\begin{tikzpicture}[scale=1]

\node (ro) at (0,0.6) {$*$};
\node (e0) at (0,0) [ext] {$1$};

\node (e1) at (-0.5,-1) [ext] {$2$};
\node (e2) at (0.5,-1) [ext] {$3$};

\draw (e0) -- +(0,0.7);
\draw (e1) -- (e0);
\draw (e2) -- (e0);

\draw[ultra thick] (e0) --  ($(e0)!0.4cm!(e2)$);
\draw[ultra thick] (e1) --  ($(e1)!0.4cm!(e0)$);
\draw[ultra thick] (e2) --  ($(e2)!0.4cm!(e0)$);

\end{tikzpicture}
\begin{tikzpicture}
\node at (0,0) {$=\displaystyle\sum_{\text{ways of connecting}}$};
\end{tikzpicture}
\begin{tikzpicture}

\node (ro) at (0,0.6) {$*$};

\node (e4) at (-1,-2) [ext] {$4$};
\node (e5) at (1,-2) [ext] {$5$};

\draw [dashed] (0,-0.5) circle (0.75);

\node (e1) at (-0.3,-0.95) [ext] {$1$};

\node (e3) at (-0.25,-0.15) [ext] {$3$};
\node (e2) at (0.35,-0.4) [ext] {$2$};

\draw (0,0.25) -- (0,0.6);
\draw (e1) -- (e2);
\draw (e1) -- (e3);

\draw (e5)  to  (0.5,-1.1);
\draw (e4)  to  (-0.42,-1.12);

\draw[ultra thick] (e2) --  ($(e2)!0.4cm!(e1)$);
\draw[ultra thick] (e3) --  ($(e3)!0.35cm!(e1)$);
\draw[ultra thick] (e1) --  ($(e1)!0.35cm!(e3)$);

\draw[ultra thick] (e4) --  ($(e4)!0.4cm!(-0.42,-1.12)$);
\draw[ultra thick] (e5) --  ($(e5)!0.4cm!(0.5,-1.1)$);

\end{tikzpicture}

%
%

\end{example}

The operad is generated by $T^i_n$, $i = 1,\dots, n$ and $T^{i,i+1}_n$, $i=1,\dots, n$, see Figure \ref{generators of CPT}.

\begin{figure}[ht]\label{generators of CPT}
 \begin{tikzpicture}[scale=1]

\node (ro) at (1,1.7) {$*$};
\node (e0) at (1,1) [ext] {$1$};
\node (e1) at (0,0) [ext] {$2$};
\node (e2) at (.75,0) [ext] {$3$};
\node (ed) at (1.3,0) {$\dots$};
\node (e3) at (2,0) [ext] {\small $n$};
\draw (e0) -- +(0,.7);
\draw[-] (e1) edge (e0) (e2) edge (e0) (e3) edge (e0);
\draw[line width=1.8] (e0) -- +(90:0.5cm);

\draw[ultra thick] (e1) --  ($(e1)!0.4cm!(e0)$);
\draw[ultra thick] (e2) --  ($(e2)!0.4cm!(e0)$);
\draw[ultra thick] (e3) --  ($(e3)!0.4cm!(e0)$);

\end{tikzpicture}\quad
 \begin{tikzpicture}[scale=1]

\node (ro) at (0.25,1.6) {$*$};
\node (e0) at (0.25,1) [ext] {$1$};
\node (e1) at (-1,0) [ext] {$2$};
\node (e2) at (-0.25,0) [ext] {$3$};
\node  at (0.25,0) {$\dots$};
\node (ei) at (0.75,0) [ext] {\small $i$};
\node (ed) at (1.25,0) {$\dots$};
\node (en) at (1.75,0) [ext] {\small $n$};
\draw (e0) -- +(0,.6);
\draw[-] (e1) edge (e0) (e2) edge (e0) (ei) edge (e0) (en) edge (e0);
\draw[ultra thick] (e0) --  ($(e0)!0.4cm!(ei)$);

\draw[ultra thick] (e1) --  ($(e1)!0.4cm!(e0)$);
\draw[ultra thick] (e2) --  ($(e2)!0.4cm!(e0)$);
\draw[ultra thick] (en) --  ($(en)!0.4cm!(e0)$);
\draw[ultra thick] (ei) --  ($(ei)!0.4cm!(e0)$);

\end{tikzpicture} \quad
 \begin{tikzpicture}[scale=1]

\node (ro) at (0.25,1.6) {$*$};
\node (e0) at (0.25,1) [ext] {$1$};
\node (e1) at (-1,0) [ext] {$2$};
\node  at (-0.5,0) {$\dots$};
\node (ei) at (0,0) [ext] {\small $i$};
\node (eii) at (0.75,0) [ext] {\tiny{$i+1$}};
\node (ed) at (1.3,0) {$\dots$};
\node (en) at (1.75,0) [ext] {\small $n$};
\draw (e0) -- +(0,.6);
\draw[-] (e1) edge (e0) (ei) edge (e0) (eii) edge (e0) (en) edge (e0);
\draw[ultra thick] (e0) -- +(-80:0.5cm);

\draw[ultra thick] (ei) --  ($(ei)!0.4cm!(e0)$);
\draw[ultra thick] (e1) --  ($(e1)!0.4cm!(e0)$);
\draw[ultra thick] (eii) --  ($(eii)!0.5cm!(e0)$);
\draw[ultra thick] (en) --  ($(en)!0.4cm!(e0)$);

\end{tikzpicture}
\caption{$T^1_n$, $T^{i}_n$ and $T^{i,i+1}_n$, from left to right.}\label{generators of CPT}
\end{figure}
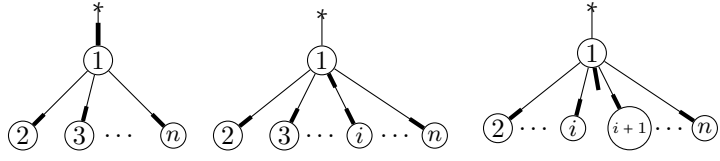

\subsection{Algebras over $\CPT$}\label{bimodule out of CSC}
The operad $\CPT$ acts naturally on spaces with cyclic structure.

\begin{proposition}
Let $\cP$ be an operad of Cyclic Swiss Cheese type. Its total space, $\prod_n \cP^2(\cdot,n)[-n]$ forms a $\CPT-\cP^1$-bimodule.
\end{proposition}
\begin{proof}
To describe the left action of $\CPT$ we use the following multi-insertion notation:

For $p_1,p_2,\dots, p_n \in \cP^2$, $p_1$ in arity $N$, we say that $I$ is a planar insertion of $p_2,\dots,p_n$ in $p_1$ if $I$ is an $N$-uple containing each $p_2,\dots,p_n$ exactly once, in that order and the other entries are filled with $\id_{\cP^2}$. For $i=1,\dots,n$, we define $i(I)$ as the position of $p_i$ in $I$. By $p_1(I)$, we mean the operadic composition given by $I$ (ignoring insertions in color $1$).

The action of $T^1_n \in \CPT$ is given by braces operations, i.e., $T^1_n(p_1, p_2,\dots,p_n) = p_1 \{p_2,\dots,p_n\}$.
The action of  $T^i_n \in \CPT$, for $i=1,\dots,n$ is given by a composition of the braces operation and a permutation of $C_{N+1}$ ``turning the mark in the direction of the root". Explicitly, if $\sigma$ is the generator of the cyclic group, $T^i_n(p_1,p_2,\dots, p_n) = \displaystyle\sum_I p_1^{\sigma^{-i(I)}}(I)$, where the sum runs over all possible planar insertions $I$ of $p_2,\dots, p_n$ in $p_1$. 

The action of $T^{i,i+1}_n$ is given by $T^{i,i+1}_n(p_1,p_2,\dots,p_n) = \displaystyle\sum_I \sum_{k=i(I)}^{(i+1)(I)} \F( p_1^{\sigma^{-k}}) (I)$, where the first sum runs over all possible planar insertions $I$ of $p_2,\dots, p_n$ in $p_1$.
This corresponds to the insertion of the element $\mathbbm 1_{\cP}$ in the marked space and the permutation sending the mark back to the direction of the root.

\end{proof}


\begin{lemma}\label{CSC => CPT} A morphism of Cyclic Swiss Cheese type operads induces a morphism of bimodules. \end{lemma}

\begin{proof}
Since a  morphism of Cyclic Swiss Cheese type operads is in particular a morphism of colored operads, the induced map on the total space is a morphism of right bimodules.
Since the definition of the action of $\CPT$ uses only the cyclic action and $\F$ and by hypothesis a morphism of Cyclic Swiss Cheese type operads commutes with these maps, the induced map on the total spaces is a left module morphism.
\end{proof}

\subsection{The operad $\CBr$}\label{Def of CBr}
We now finish the construction of the Cyclic Braces operad via operadic twisting. 
There is a map $F\colon \Lie\{1\} \to \CPT$ sending the Lie bracket to

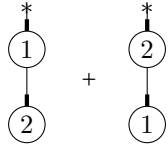
\[
 \begin{tikzpicture}[scale=1,
vert/.style={draw,outer sep=0,inner sep=0,minimum size=5,shape=circle,fill},
helper/.style={outer sep=0,inner sep=0,minimum size=5,shape=coordinate},
default_edge/.style={draw,-},
ext/.style={draw,outer sep=0,inner sep=2,minimum size=5,shape=circle},
every loop/.style={}]

\node (star) at (5,8) {*};
\node (star2) at (6.6,8) {*};
\node (v0) at (5,6.5) [ext] {2};
\node (v1) at (5,7.5) [ext] {1};
\node (v3) at (5,8.4) [helper] {};
\node (v10) at (6.6,7.5) [ext] {2};
\node (v9) at (6.6,6.5) [ext] {1};
\node (v12) at (6.6,8.4) [helper] {};
\node (v14) at (5.6,7.1) [helper,label=0:{$+$}] {};

\draw (v1) -- +(0,.6);
\draw (v10) -- +(0,.6);

\draw[default_edge] (v0) to (v1);
\draw[default_edge] (v9) to (v10);

\draw[ultra thick] (v1) --  ($(v1)!0.4cm!(star)$);
\draw[ultra thick] (v10) --  ($(v10)!0.4cm!(star2)$);
\draw[ultra thick] (v0) --  ($(v0)!0.4cm!(star)$);
\draw[ultra thick] (v9) --  ($(v9)!0.4cm!(star2)$);
\end{tikzpicture}
\]

Using $F$ we consider the (dg) operad given by the operadic twisting of $\CPT$, $ Tw \CPT$ (see the Appendix for details).

The space $Tw \CPT(n) =\left(\displaystyle\prod_k \CPT(n+k) \otimes \mathbb K [-2]^{\otimes k}\right)_{\mathbb S_k}$ is made out of trees, similar to the ones in $\CPT$ but with vertices of two different kinds. There are $n$ \textit{external} vertices, labeled from $1$ to $n$ and $k$ \textit{internal} unlabeled vertices, that we draw as a full black vertex. The degree of each edge or marked space is $-1$, the degree of an external vertex is $0$ and the degree of an internal vertex is $2$.

This operad is generated by elements as in Figure \ref{generators of CPT} together with $T'^i_n$and $T'^{i,i+1}_n$, $i=0,\dots,n$: 
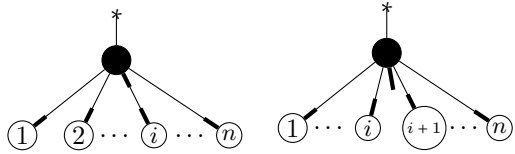
\begin{figure}[h]
 \begin{tikzpicture}[scale=1]

\node (ro) at (0.25,1.6) {$*$};
\node (e0) at (0.25,1) [int] {1};
\node (e1) at (-1,0) [ext] {$1$};
\node (e2) at (-0.25,0) [ext] {$2$};
\node  at (0.25,0) {$\dots$};
\node (ei) at (0.75,0) [ext] {\small $i$};
\node (ed) at (1.25,0) {$\dots$};
\node (en) at (1.75,0) [ext] {\small $n$};
\draw (e0) -- +(0,.6);
\draw[-] (e1) edge (e0) (e2) edge (e0) (ei) edge (e0) (en) edge (e0);
\draw[ultra thick] (e0) --  ($(e0)!0.4cm!(ei)$);

\draw[ultra thick] (e1) --  ($(e1)!0.4cm!(e0)$);
\draw[ultra thick] (e2) --  ($(e2)!0.4cm!(e0)$);
\draw[ultra thick] (en) --  ($(en)!0.4cm!(e0)$);
\draw[ultra thick] (ei) --  ($(ei)!0.4cm!(e0)$);

\end{tikzpicture} \quad
 \begin{tikzpicture}[scale=1]

\node (ro) at (0.25,1.6) {$*$};
\node (e0) at (0.25,1) [int] {1};
\node (e1) at (-1,0) [ext] {$1$};
\node  at (-0.5,0) {$\dots$};
\node (ei) at (0,0) [ext] {\small $i$};
\node (eii) at (0.75,0) [ext] {\tiny{$i+1$}};
\node (ed) at (1.3,0) {$\dots$};
\node (en) at (1.75,0) [ext] {\small $n$};
\draw (e0) -- +(0,.6);
\draw[-] (e1) edge (e0) (ei) edge (e0) (eii) edge (e0) (en) edge (e0);
\draw[ultra thick] (e0) -- +(-80:0.5cm);

\draw[ultra thick] (ei) --  ($(ei)!0.4cm!(e0)$);
\draw[ultra thick] (e1) --  ($(e1)!0.4cm!(e0)$);
\draw[ultra thick] (eii) --  ($(eii)!0.5cm!(e0)$);
\draw[ultra thick] (en) --  ($(en)!0.4cm!(e0)$);
\end{tikzpicture}
\end{figure}

The differential has two pieces, the first is computed by taking the operadic Lie bracket with 
 \begin{tikzpicture}[scale=1]

\node (ro) at (1,1.7) {$*$};
\node (e0) at (1,1) [int] {$1$};
\node (e1) at (1,0) [ext] {$1$};

\draw (e1) -- (e0) (e0)--+(0,.7);
\draw[ultra thick] (e1) --  ($(e1)!0.4cm!(e0)$);
\draw[ultra thick] (e0) --  ($(e0)!0.4cm!(ro)$);

\node at (1.5,0.65) {$+$};

\node (ro2) at (2,1.7) {$*$};
\node (e02) at (2,1) [ext] {$1$};
\node (e12) at (2,0) [int] {$1$};

\draw (e12) -- (e02) (e02)--+(0,.7);
\draw[ultra thick] (e12) --  ($(e12)!0.4cm!(e02)$);
\draw[ultra thick] (e02) --  ($(e02)!0.4cm!(ro2)$);

\end{tikzpicture} $= T'^1_1 + T^1_2 \circ_1 T'^1_0$ , which amounts to split an internal vertex at every external vertex, but subtracting some combinations with one 1-valent or 2-valent internal vertex. The second piece just splits an internal vertex out of every internal vertex.

\begin{lemma}
The subspace $(Tw \CPT)'\subset Tw \CPT$ spanned by trees whose internal vertices are at least 3-valent is a suboperad of $Tw \CPT$.
\end{lemma}

\begin{proof}
The composition of trees in $(Tw \CPT)'$ cannot create internal vertices with valence $1$ or $2$. 

The differential, however can create both kinds of vertices, so we must check that these contributions are canceled. 

$1$-valent internal vertices can be created at every internal vertex by splitting it and reconnecting all edges incident edges to one of the internal vertices. Similarly, $1$-valent internal vertices can be created at an external vertex when inserting  $T'^1_1 + T^1_2 \circ_1 T'^1_0$ at that vertex and then reconnect to the external vertex. These contributions are all canceled by the remaining term of the differential consisting in inserting the tree in  $T'^1_1 + T^1_2 \circ_1 T'^1_0$.

To see that $2$-valent internal vertices contributions are canceled, it is enough to notice that every time such a vertex is created, it will be canceled by a similar contribution on the other adjacent vertex.
\end{proof}

\begin{defi}
We define the Cyclic Braces operad as $\CBr:= (Tw \CPT)'/J$, where $J$ is the operadic ideal generated by $T'^i_n - T'^{i-1}_n, \  i=0,\dots,n$ and

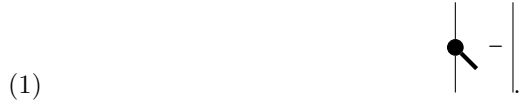
\begin{equation}\label{Ideal}
\begin{tikzpicture}[scale=1]
\node (v0) at (0,0) [int] {.};
\draw (v0) -- (0,0.6);
\draw (v0) -- (0,-0.6);
\draw[ultra thick] (v0) --  +(-45:0.4cm);
\end{tikzpicture}
\begin{tikzpicture}[scale=1]
\node at (0,0) {};
\node at (0,0.5) {$-$};
\end{tikzpicture}
\begin{tikzpicture}[scale=1]

\draw (0,0.6) -- (0,-0.6);
\end{tikzpicture}. 
\end{equation}
\end{defi}

\begin{remark}\label{homogeneous}
The $T'^i_n - T'^{i-1}_n$ in $J$ mean that in $\CBr$ the marks at internal vertices are irrelevant. We will therefore not draw them in pictures and we will denote the image of $T'^i_n$ in $\CBr$ just by $T'_n$.\end{remark}

\begin{convention}
Since $J$ is not homogeneous with respect to the number of (internal) vertices, the number of (internal and therefore the total number of) vertices of a cyclic braces tree is \textit{a priory} ill defined. We shall consider that whenever we have subsection of a tree like this \begin{tikzpicture}[scale=1]
\node (v0) at (0,0) [int] {.};
\draw (v0) -- (0,0.6);
\draw (v0) -- (0,-0.6);
\draw[ultra thick] (v0) --  +(-45:0.4cm);
\end{tikzpicture} that there is only one edge and no vertices.
\end{convention}

\subsection{The homology of the Cyclic Braces operad}
In this subsection we show that the homology of $\CBr$ is the $\BV$ operad. For this, we make use of the operad $\Br$ whose homology, as mentioned in the beginning of this section, is the operad $\Ger$.

\begin{defi}
The operad $\Br$ is defined as the suboperad of $\CBr$ generated by $T_n^1$ and $T_n'$, or equivalently, the suboperad spanned by trees whose marks at every vertex are pointing towards the root. 
\end{defi} 

In $\Br$ we ``forget" that there are marks at vertices, therefore when referring to this operad we use the notation $T_n$ instead of $T_n^1$ and we do not draw the marks in the pictures.

Two trees in $\CBr$ are said to have the same shape if when one forgets about the marks at vertices and connections to $\mathbbm 1$, they are the same. For example, $T^i_n$ and $T^{i,i+1}_n$ have the same shape.

Let let us consider the map $f= \oplus_{n}f_n\colon \Br(n)\otimes (\mathbb K \oplus \mathbb K[1])^{\otimes n}\to \CBr(n) $
 sending $T\otimes \epsilon$, where $T$ is braces tree and $\epsilon = \epsilon_1\otimes\dots\otimes\epsilon_n\in (\mathbb K \oplus \mathbb K[1])^{\otimes n}$, to the a sum of cyclic braces trees of the same shape, according to the following rules:

If the $\epsilon_i = (1,0)$, the vertex labeled by $i$ is sent to the same vertex with the marking pointing in the direction of the root. 

If the $\epsilon_i = (0,1)$, the vertex labeled by $i$ is sent to a sum over all possible ways of inserting an edge connecting to $\mathbbm 1$.

\begin{lemma}\label{H(CBR)=BV}
$f$ is a quasi-isomorphism of chain complexes.
\end{lemma}
\begin{proof}
Since marked spaces have degree $-1$, $f$ preserves the degree. Since the differential acts by derivations, it is enough to check that $f$ commutes with the differentials on every vertex $i$ and this is clearly the case if $\epsilon_i = (1,0)$. 

Let us consider the case of $T_n =$ \begin{tikzpicture}[scale=1]

\node (ro) at (1,1.7) {$*$};
\node (e0) at (1,1) [ext] {$1$};
\node (e1) at (0,0) [ext] {$2$};
\node (e2) at (.75,0) [ext] {$3$};
\node (ed) at (1.3,0) {$\dots$};
\node (e3) at (2,0) [ext] {\small $n$};
\draw (e0) -- +(0,.7);
\draw[-] (e1) edge (e0) (e2) edge (e0) (e3) edge (e0);

\end{tikzpicture} $\in \Br$ with $\epsilon_1= (0,1)$.

\begin{equation}\label{eq:dTn}
dT_n=\displaystyle\sum_{\substack{ \text{ways of} \\ \text{connecting}}} 
\begin{tikzpicture}[scale=1]
\node (ro) at (1,1.7) {$*$};

\node (e1) at (1,1) [ext] {$1$};
\node (black) at (1,0.3) [int] {$*$};

\node (e2) at (0,-1) [ext] {$2$};
\node (e3) at (.75,-1) [ext] {$3$};
\node (ed) at (1.3,-1) {$\dots$};
\node (en) at (2,-1) [ext] {\small $n$};
\draw (e1) -- +(0,.7);

\draw [dashed] (1,0.45) circle (0.75);
\draw (black) -- (e1);

\draw (e2) -- (0.5,-0.1);
\draw (e3) -- (0.85, -0.3);
\draw (en) -- (1.5 ,-0.1);

\end{tikzpicture}
+
\begin{tikzpicture}[scale=1]
\node (ro) at (1,1.7) {$*$};

\node (e1) at (1,0.3) [ext] {$1$};
\node (black) at (1,1) [int] {$*$};

\node (e2) at (0,-1) [ext] {$2$};
\node (e3) at (.75,-1) [ext] {$3$};
\node (ed) at (1.3,-1) {$\dots$};
\node (en) at (2,-1) [ext] {\small $n$};
\draw (black) -- +(0,.7);

\draw [dashed] (1,0.45) circle (0.75);
\draw (black) -- (e1);

\draw (e2) -- (0.5,-0.1);
\draw (e3) -- (0.85, -0.3);
\draw (en) -- (1.5 ,-0.1);
\end{tikzpicture},
\end{equation}

 where the sum runs over all planar possible ways of connecting the incident edges such that the internal vertex is at least trivalent.

We have $f(T_n) = \sum_{i=1}^n T_n^{i,i+1}$, following the notation in Figure \ref{generators of CPT}. If we compute $df(T_n)$, the part of the differential given by the insertion of \begin{tikzpicture}[scale=1]
\node (ro) at (1,1.6) {$*$};

\node (e1) at (1,0.3) [ext] {$1$};
\node (black) at (1,1) [int] {$*$};

\draw (black) -- (e1);
\draw (black) -- +(0,0.6);
\end{tikzpicture} on every $T_n^{i,i+1}$ is canceled over all the sum. 

Therefore $df(T_n) = \displaystyle \sum_i \sum_{\substack{ \text{ways of} \\ \text{connecting}}}$
 \begin{tikzpicture}[scale=1]

\node (ro) at (0.25,1.6) {$*$};
\node (e1) at (0.25,-0.08) [ext] {$1$};
\draw (e1) -- +(0,0.75) node (black) [int] {$a$};

\node (e2) at (-1,-1) [ext] {$2$};
\node  at (-0.5,-1) {$\dots$};
\node (ei) at (0,-1) [ext] {\small $i$};
\node (eii) at (0.75,-1) [ext] {\tiny{$i+1$}};
\node (ed) at (1.3,-1) {$\dots$};
\node (en) at (1.75,-1) [ext] {\small $n$};
\draw (0.25,1.2) -- (0.25,1.6);

\draw [dashed] (0.25,0.45) circle (0.75);

\draw (e2) -- (-0.25,-0.1);
\draw (ei) -- (0, -0.25);
\draw (eii) -- (0.6, -0.2);
\draw (en) -- (0.75 ,-0.1);

\draw[ultra thick] (e1) -- +(-90:0.5cm);

\draw[ultra thick] (ei) --  ($(ei)!0.4cm!(0, -0.25)$);
\draw[ultra thick] (e2) --  ($(e2)!0.4cm!(-0.25,-0.1)$);
\draw[ultra thick] (eii) --  ($(eii)!0.5cm!(0.6, -0.2)$);
\draw[ultra thick] (en) --  ($(en)!0.4cm!(0.75 ,-0.1)$);

\end{tikzpicture}. To see that $df(T_n) = f(dT_n)$, we note that there are two possibilities. Every summand of $df(T_n)$ has either the internal vertex connected to the root vertex or the vertex labeled by $1$ connected to the root vertex. If the root is connected to the internal vertex, we find that same summand on the image by $f$ of the second type of trees on equation \eqref{eq:dTn}, and similarly if the root is connected to the vertex $1$.

Conversely, all trees that we get when we compute $f(dT_n)$ appear only once (due to the planar ordering of edges and marks around a vertex) and can be obtained as a summand in $df(T_n)$.

To show that $f$ is a quasi-isomorphism, we filter $\CBr$ and $\Br$ by the number of internal vertices (see Remark \ref{homogeneous}). The map $f$ is compatible with these filtrations and on the zeroth page of the corresponding spectral sequence in $\CBr$ one obtains the only piece of the differential that does not increase the number of internal vertices. Explicitly $d_0 (T^{i,i+1}_n) = T^{i+1}_n-T^{i}_n$ and $d_0(T^{j}_n)=0$. On the correspondent spectral sequence in $\Br$ one obtains the zero differential.

The differential $d_0$ respects the shape of the tree. Therefore the complex $(\CBr,d_0)$ splits as $$\CBr (n) = \bigoplus_{\text{Shape } S} V_S,$$
where the sum runs over all possible shapes $S$ of trees with $n$ external vertices and $V_S$ is the subcomplex spanned by all trees with the shape $S$.

The differential acts on the tree by acting on every vertex by means of the Leibniz rule, therefore if $V_S^i$ represents the space of the $i$-th vertices of the trees with the given shape, then each $V_S$ splits as a complex as $V_S = \bigotimes_{i=1}^n V_S^i$ (up to some degree shift).

But $(V^i_S,d_0)$ is isomorphic to the simplicial complex of the $k$-gon, where $k$ is the valence of the vertex $i$ (again, up to some degree shift).

Therefore $H(\CBr, d_0) =\displaystyle\bigoplus_{\text{Shape } S} \left( \bigotimes_{i=1}^n H(V_S^i)\right) [k_S]= \bigoplus_{\text{Shape } S} \left(\bigotimes_{i=1}^n (\mathbb K \oplus \mathbb K[1])\right) [k_S]$, where $k_S$ is a degree shift dependent only on the shape of the tree. 

Then, at the level of the homology of the zeroth pages of the spectral sequences we get and induced map $\Br(n)\otimes (\mathbb K \oplus \mathbb K[1])^{\otimes n}\to \displaystyle\bigoplus_{\text{Shape } S} (\mathbb K \oplus \mathbb K[1])^{\otimes n} [k_S]$.

Since clearly every possible shape of Cyclic Braces trees has a unique representative that is a Braces tree, this induced map is an isomorphism. Therefore $f$ induces a quasi-isomorphism on the zeroth page of the spectral sequence, which implies that $f$ is a quasi-isomorphism between the original complexes.

\end{proof}

\begin{corollary}
The homology of $\CBr$ is $\BV$, the operad governing BV algebras.
\end{corollary}

\begin{proof}

As a consequence of Lemma \ref{H(CBR)=BV} we have $H(\CBr(n)) = H(\Br(n)\otimes (\mathbb K \oplus \mathbb K[1])^{\otimes n}) = H(\Br(n)) \otimes   (\mathbb K \oplus \mathbb K[1])^{\otimes n} = \Ger(n) \otimes (\mathbb K \oplus \mathbb K[1])^{\otimes n} \cong \BV(n)$.

Keeping track of signs, it is easy to check that  \begin{tikzpicture}[scale=0.8]

\node (ro) at (1,1.7) {$*$};
\node (e0) at (1,1) [int] {$*$};
\node (e1) at (0.5,0) [ext] {$1$};
\node (e2) at (1.5,0) [ext] {$2$};
\draw (e0) -- +(0,.7);
\draw[-] (e1) edge (e0) (e2) edge (e0);


\draw[ultra thick] (e1) --  ($(e1)!0.5cm!(e0)$);
\draw[ultra thick] (e2) --  ($(e2)!0.5cm!(e0)$);
\end{tikzpicture}, \begin{tikzpicture}[scale=1]
\node (v0) at (5,6.7) [ext] {$2$};
\node (v1) at (5,7.5) [ext] {$1$};
\node (v3) at (5,8.1)  {$*$};
\node (v10) at (6.2,7.5) [ext] {2};
\node (v9) at (6.2,6.7) [ext] {1};
\node (v12) at (6.2,8.1)  {*};
\node (v14) at (5.6,7.1)  {+};

\draw (v1) -- +(0,.6);
\draw (v10) -- +(0,.6);

\draw (v0) -- (v1);
\draw (v9) -- (v10);

\draw[ultra thick] (v0) --  ($(v0)!0.4cm!(v1)$);
\draw[ultra thick] (v1) --  ($(v1)!0.4cm!(v3)$);

\draw[ultra thick] (v10) --  ($(v10)!0.4cm!(v12)$);
\draw[ultra thick] (v9) --  ($(v9)!0.4cm!(v10)$);
\end{tikzpicture} and \begin{tikzpicture}[scale=1]
\node (v0) at (0,0) [ext] {1};
\node (ro) at (0,0.6) {*};

\draw (v0) -- (0,0.6);
\draw[ultra thick] (v0) --  +(-45:0.5cm);
\end{tikzpicture} satisfy, up to homotopy, the relations of $\cdot$, $[\ ,\ ]$ and $\Delta$, the generators of $\BV$.

For example, the equality

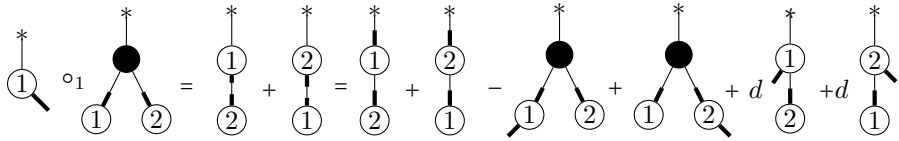
\begin{tikzpicture}[scale=1]
\node  at (0.2,0) {\begin{tikzpicture}[scale=1]
\node (v0) at (0,0) [ext] {1};
\node (ro) at (0,0.6) {*};

\draw (v0) -- (0,0.6);
\draw[ultra thick] (v0) --  +(-45:0.5cm);
\end{tikzpicture}};

\node at (0.8,-0.2) {$\circ_1$};

\node at (1.5,0) {\begin{tikzpicture}[scale=0.8]

\node (ro) at (1,1.7) {*};
\node (e0) at (1,1) [int] {$*$};
\node (e1) at (0.5,0) [ext] {$1$};
\node (e2) at (1.5,0) [ext] {$2$};
\draw (e0) -- +(0,.7);
\draw[-] (e1) edge (e0) (e2) edge (e0);


\draw[ultra thick] (e1) --  ($(e1)!0.5cm!(e0)$);
\draw[ultra thick] (e2) --  ($(e2)!0.5cm!(e0)$);
\end{tikzpicture}};

\node at (2.3,-0.3) {$=$};

\node at (3.4,0) {\begin{tikzpicture}[scale=1]
\node (v0) at (5,6.7) [ext] {$2$};
\node (v1) at (5,7.5) [ext] {$1$};
\node (v3) at (5,8.1)  {*};

\node (v10) at (6,7.5) [ext] {2};
\node (v9) at (6,6.7) [ext] {1};
\node (v12) at (6,8.1)  {*};
\node (v14) at (5.5,7.1)  {$+$};

\draw (v1) -- +(0,.6);
\draw (v10) -- +(0,.6);

\draw (v0) -- (v1);
\draw (v9) -- (v10);

\draw[ultra thick] (v0) --  ($(v0)!0.35cm!(v1)$);
\draw[ultra thick] (v1) --  ($(v1)!0.3cm!(v0)$);

\draw[ultra thick] (v10) --  ($(v10)!0.35cm!(v9)$);
\draw[ultra thick] (v9) --  ($(v9)!0.3cm!(v10)$);
\end{tikzpicture}};

\node at (4.35,-0.3) {$=$};

\node at (5.3,0) {\begin{tikzpicture}[scale=1]
\node (v0) at (5,6.7) [ext] {$2$};
\node (v1) at (5,7.5) [ext] {$1$};
\node (v3) at (5,8.1)  {*};
\node (v10) at (6,7.5) [ext] {2};
\node (v9) at (6,6.7) [ext] {1};
\node (v12) at (6,8.1)  {*};
\node (v14) at (5.5,7.1)  {$+$};

\draw (v1) -- +(0,.6);
\draw (v10) -- +(0,.6);

\draw (v0) -- (v1);
\draw (v9) -- (v10);

\draw[ultra thick] (v0) --  ($(v0)!0.4cm!(v1)$);
\draw[ultra thick] (v1) --  ($(v1)!0.4cm!(v3)$);

\draw[ultra thick] (v10) --  ($(v10)!0.4cm!(v12)$);
\draw[ultra thick] (v9) --  ($(v9)!0.4cm!(v10)$);
\end{tikzpicture}};

\node at (6.4,-0.3) {$-$};

\node at (7.2,0)  {\begin{tikzpicture}[scale=0.8]

\node (ro) at (1,1.7) {*};
\node (e0) at (1,1) [int] {$*$};
\node (e1) at (0.5,0) [ext] {$1$};
\node (e2) at (1.5,0) [ext] {$2$};
\draw (e0) -- +(0,.7);
\draw[-] (e1) edge (e0) (e2) edge (e0);


\draw[ultra thick] (e1) --  ($(e1)!0.5cm!(e0)$);
\draw[ultra thick] (e2) --  ($(e2)!0.5cm!(e0)$);
\draw[ultra thick] (e1) --  +(-135:0.5cm);

\end{tikzpicture}};

\node at (8.0,-0.3) {$+$};

\node at (8.9,0)  {\begin{tikzpicture}[scale=0.8]

\node (ro) at (1,1.7) {*};
\node (e0) at (1,1) [int] {$*$};
\node (e1) at (0.5,0) [ext] {$1$};
\node (e2) at (1.5,0) [ext] {$2$};
\draw (e0) -- +(0,.7);
\draw[-] (e1) edge (e0) (e2) edge (e0);


\draw[ultra thick] (e1) --  ($(e1)!0.5cm!(e0)$);
\draw[ultra thick] (e2) --  ($(e2)!0.5cm!(e0)$);
\draw[ultra thick] (e2) --  +(-45:0.5cm);

\end{tikzpicture}};

\node at (9.7,-0.3) {$+ \ d$};

\node at (10.3,0) {\begin{tikzpicture}[scale=1]
\node (v0) at (5,6.7) [ext] {$2$};
\node (v1) at (5,7.5) [ext] {$1$};
\node (v3) at (5,8.1)  {$*$};

\draw (v1) -- +(0,.6);

\draw (v0) -- (v1);

\draw[ultra thick] (v0) --  ($(v0)!0.4cm!(v1)$);
\draw[ultra thick] (v1) --  +(-125:0.4cm);
\end{tikzpicture}};

\node at (10.9,-0.3) {$+ d$};

\node at (11.5,0) {\begin{tikzpicture}[scale=1]
\node (v0) at (5,6.7) [ext] {$1$};
\node (v1) at (5,7.5) [ext] {$2$};
\node (v3) at (5,8.1)  {*};

\draw (v1) -- +(0,.6);

\draw (v0) -- (v1);

\draw[ultra thick] (v0) --  ($(v0)!0.4cm!(v1)$);
\draw[ultra thick] (v1) --  +(-45:0.4cm);
\end{tikzpicture}};

\end{tikzpicture}

corresponds in homology to the equation $\Delta \circ \cdot = [\ ,\ ] + \cdot \circ_1 \Delta + \cdot \circ_2 \Delta$.

Therefore, since the dimensions in every arity are the same (and finite), the operad $H(\CBr)$ is canonically isomorphic to $\BV$.
\end{proof}

\section{Operadic bimodule maps}\label{section: main thm in Rd}

 Given an operad $\cP$ and a resolution $\cP_\infty \aol \cP_\infty^{\text{bimod}} \aor \cP_\infty$ of the canonical bimodule $\cP \aol \cP \aor \cP$, an infinity morphism of $\cP_\infty$ algebras $A$ and $B$, can be expressed as the following bimodule map:

\begin{equation*}
\begin{tikzcd}[column sep=0.5em]				
\cP_\infty \arrow{d}{} &\aol &  \cP_\infty^{\text{bimod}} \arrow{d}{} & \aor &\cP_\infty^{\text{bimod}}\arrow{d}{} \\
\End B & \aol  & \Hom(A^{\otimes \bullet}, B)  &\aor & \End A,\\
\end{tikzcd}
\end{equation*}
where by $\End A$ we mean the operadic endomorphisms $\End A(n) = \Hom(A^\otimes n, A)$ and the bimodule structure on $\Hom(A^{\otimes \bullet}, B)$ is the natural one using composition of maps.
In this section we prove Theorem \ref{main theorem} by expressing it in terms of a morphism of bimodules.

\subsection{$Chains(\mathsf{F}\mathbb H_{m,n})\to \BVKGra$}\label{subsec: chains to graphs}
The topological operad of Cyclic Swiss Cheese type $\left(\FFM_2,\mathsf F \mathbb H_{m,n}\right)$ introduced in \ref{FFM} is in fact an operad on the category of semi-algebraic manifolds\cite{semi-algebraic manifolds,discs}. We consider the functor $Chains$ of semi-algebraic chains. This functor is monoidal so it induces a functor from semi-algebraic Cyclic Swiss Cheese type operads to dg Cyclic Swiss Cheese type operads.

In this section we define a morphism of Cyclic Swiss Cheese type operads 
\begin{equation}\label{chains to graphs}
\left(Chains(\FFM_2),Chains(\mathsf F \mathbb H_{m,n})\right) \to \left(\BVGra,\BVKGra\right).
\end{equation}

We start by defining a map $f_2\colon\BVKGra^*\to \Omega(\mathsf F \mathbb H_{m,n})$, where $\Omega$ is the functor sending a semi-algebraic manifold to its algebra of semi-algebraic forms.

Notice that $\mathsf F \mathbb H_{m,n}$ is a quotient of the configuration space of $m$ points in the upper half plane and $n$ points at the boundary by a group of conformal maps. The identification of $\mathbb H$ with the Poincar\'e Disk necessary for the definition of the cyclic action and the forgetful map is also conformal. Therefore, given a point $p$ in the upper half plane and a point $q$ either in the upper half plane or at the boundary  the angle between the hyperbolic line passing by the point at $\infty$ and $p$ and the hyperbolic line passing by the points $p$ and $q$ is well defined (up to a multiple of $2\pi$).

We define $d\phi^i_j \in \Omega^1(\mathsf F \mathbb H_{m,n})$, for $1\leq i\leq m$ and $1\leq j\leq n$ as the 1-form given by the angle made by the hyperbolic line defined by the point at $\infty$ and the point labeled by $i$ and the hyperbolic line defined by the point labeled by $i$ and the point labeled by $\overline j$.

 Similarly, $1\leq i\ne j\leq m$, we define $d\phi^{i,j} \in \Omega^1(\mathsf F \mathbb H_{m,n})$ as the 1-form given by the angle defined by the line passing by $\infty$ and $i$ and the line passing by $i$ and $j$.

Finally, we define $d\phi^{i,i}\in\Omega^1(\mathsf F \mathbb H_{m,n})$ as the 1-form corresponding to the angle between the line passing by $\infty$ and $i$ and the frame at $i$.

%
%
%
%
%
%
%

\begin{figure}[h]
\begin{tikzpicture}[scale=2]
  \tkzDefPoint(0,0){O}
  \tkzDefPoint(1,0){A}
  \tkzDrawCircle(O,A)
  \tkzDefPoint(0.3,-0.25){j}
  \tkzDefPoint(-0.5,-0.5){i}
    \tkzDefPoint(0,1){z3}
  \tkzClipCircle(O,A)
  \tkzDrawCircle[orthogonal through=i and j](O,A)
    \tkzDrawCircle[orthogonal through=i and z3](O,A)
  \tkzDrawPoints[color=black,fill=black,size=12](i,j)
  \tkzLabelPoints(i,j)
  
  \fill (0,1) circle[radius=1.2pt] node[below left] {$\infty$};

  \node at (0.1,0) {$\phi^{i,j}$}; 
  \draw (-0.21,-0.03) arc (90:-20:0.2);
\end{tikzpicture}
\caption{The hyperbolic angle $\phi^{i,j}$}
\end{figure}
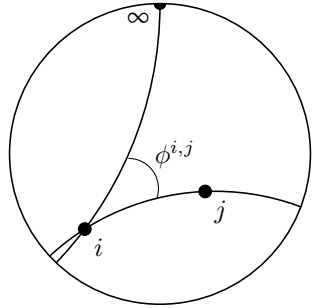

There is a canonical basis of $\BVKGra(m,n)$ given by the graphs and, by abuse of notation, we denote by the same graph the dual basis of $\BVKGra^*(m,n)$

Following the notation in $\ref{BVKGra}$, we define  $f_2(\Gamma^i_j) := \frac{d\phi^i_j}{2\pi}$ for $1\leq i \leq m$, $1\leq j \leq n$ and $f_2(\Gamma^{i,j}):=\frac{d\phi^{i,j}}{2\pi}$ for $i\ne j$ between $1$ and $m$.

$\BVKGra^*(m,n)$ admits a similar algebra structure by defining the product of two graphs as the superposition of edges. 
We extend the map $f_2$ to $\BVKGra^*$ by requiring it to be a morphism of unital algebras.

A $C_{n+1}$ action on $\BVKGra^*(m,n)$ can be defined via the pullback of the cyclic action on $\BVKGra(m,n)$. Notice that this is not the standard definition of an action of a group on the dual space (one normally uses the pullback via the inverse of the map), but since $C_{n+1}$ is abelian no problems arise from this.

$\Omega(\mathsf F \mathbb H_{m,n})$ inherits a $C_{n+1}$ cyclic action from the cyclic action in $\mathsf F \mathbb H_{m,n}$ (also by pullback). 

\begin{lemma}\label{gratoforms is equiv}
The map $f_2 \colon \BVKGra^*(m,n)\to \Omega(\mathsf F \mathbb H_{m,n})$ is $C_{n+1}$ equivariant.
\end{lemma}
\begin{proof}
Notice actually that the algebra structure on $\BVKGra(m,n)$ is in fact the exterior algebra $\bigwedge V$, where $V$ is the (finite dimensional) vector space concentrated in degree $-1$ spanned by all graphs with exactly one edge. 

We had defined the cyclic action on $V$, extended this action to $\bigwedge V$ by requiring the action to commute with the product and defined an action on $(\bigwedge V)^* = \BVKGra(m,n)$. Alternatively, the cyclic action on $V$ induces a cyclic action on $V^*$ which induces a cyclic action on $\bigwedge V^*$. Under the identification $\bigwedge V^*=(\bigwedge V)^*$ these two actions are the same. This is an immediate consequence of the fact that if $e_1,\dots,e_n$ are part of a basis of $V$ and $e^*_1,\dots,e^*_n$ are the corresponding parts of the dual basis, then $e_1\wedge\dots \wedge e_n$ is dual to $e^*_1\wedge\dots\wedge e^*_n$.

This allows us to conclude that the cyclic action on $\BVKGra^*(m,n)$ commutes with the product of graphs.

It is therefore enough to show that $f_2$ is equivariant with respect to one-edge graphs.

The cyclic action of $C_{n+1} = \langle\sigma\rangle$ on one-edge graphs in $\BVKGra^*(m,n)$ is given by $\left(\Gamma^{i,j}\right)^\sigma =\Gamma^{i,j} -\Gamma^i_1$ and $\left(\Gamma^i_j\right)^{\sigma} = \Gamma^i_{j+1}-\Gamma^i_1$ with the convention that $\Gamma^i_{n+1}=0$.

Since the cyclic action on $\mathsf F \mathbb H_{m,n}$ is by rotation of the $n$ points at the boundary with the point $\infty$, we have $\left(d\phi^{i}_j\right)^{\sigma} = d(\phi^{i}_j\cdot {\sigma}) =  d(\phi^{i}_{j+1} - \phi^{i}_1)$ and similarly $\left(d\phi^{i,j}\right)^{\sigma} = \left(d\phi^{i,j} - d\phi^i_1\right)$, therefore $f_2$ commutes with the action.
\end{proof}

Analogously, a map $f_1 \colon \BVGra^*(n) \to \Omega( \FFM_2)(n)$ can be defined on one-edge graphs by considering the angle with the vertical and extending as a morphism of algebras.

\begin{remark}\label{cooperadmap}
It is easy to check on generators that these maps
produce a map of colored cooperads\footnote{Strictly speaking, the right hand side is not a cooperad, but this does not affect what follows.}
$$(f_1,f_2) \colon \left(\BVGra^*,\BVKGra^*\right) \to\left(\Omega(\FFM_2),\Omega(\mathsf F \mathbb H_{m,n})\right).$$ 

Let us sketch the verification for the case of $\Gamma^{1,2}\in \BVKGra^*(2,0)$.

The composition map in $\left(\FFM_2, \mathsf F\mathbb H\right)$ is done by insertion at the boundary stratum with an appropriate rotation given by the framing. Since the cocomposition map is given by the pullback of the composition map, the part of the cocomposition given by  $\Omega(\mathsf F \mathbb H) \to \Omega(\mathsf F \mathbb H) \otimes \bigotimes \Omega(\FFM_2)$ sends $d\phi^{1,2} \in \mathsf F \mathbb H(2,0)$ to $d\phi^{1,1}\otimes 1+1\otimes d\phi^{1,2} \in \Omega(\mathsf F \mathbb H(1,0)) \otimes \Omega(\FFM_2(2))$ (recall Figure \ref{composition ffm}).

The corresponding cocomposition in $\BVKGra^*$ sends $\Gamma^{1,2}$ to 

$\left(\Gamma^{1,1} \otimes 
\begin{tikzpicture}[scale=0.5]
\node[label=$1$] at (0,0) [int]{};
\node[label=$2$] at (1,0) [int]{};
\end{tikzpicture}\right)$
$+ \left(1\otimes
\begin{tikzpicture}[scale=0.5]
\node[label=$1$] (1) at (0,0) [int]{};
\node[label=$2$] (2) at (1,0) [int]{};
\draw [->] (1)--(2);
\end{tikzpicture}\right) \in \BVKGra^*(2,0) \otimes \BVGra^*(2)$, therefore the diagrams commute. The general case for $\Gamma^{i,j}\in \BVKGra^*(m,n)$ is similar and all the remaining cases are as simple or even simpler to check.
\end{remark}

We define a map $g_1\colon Chains(\FFM_2) \to \Omega^*(\FFM_2)$ that maps every elementary semi-algebraic chain $c\in Chains(\FFM_2)$ to the linear form
$\omega \mapsto \int_c \omega$. Similarly we define $g_2\colon Chains(\mathsf F \mathbb H) \to \Omega^*(\mathsf F \mathbb H)$ sending a chain to integration over that chain.

Clearly $\BVKGra(m,n)$ is finite dimensional for a fixed degree, therefore its double dual of $\BVKGra(m,n)$ can be identified with the original space.

Finally, the map of Cyclic Swiss Cheese type operads \eqref{chains to graphs} that we were searching is defined as the composition
$$\left(Chains(\FFM_2),Chains(\mathsf F \mathbb H)\right)\xrightarrow{(g_1,g_2)}\left(\Omega^*(\FFM_2),\Omega^*(\mathsf F \mathbb H)\right) \xrightarrow{(f_1^*,f_2^*)} \left(\BVGra,\BVKGra\right).$$

This is a colored operad map as a consequence of Remark \ref{cooperadmap}, it commutes with the cyclic action as a consequence of Lemma \ref{gratoforms is equiv} and by hand one checks that $\mathbbm 1_{Chains(\mathsf F \mathbb H)}$ is sent to $\mathbbm 1_{\BVKGra}$.

Explicitly, given a chain $c\in Chains(\FFM_2)$, we have
$f_1^* \circ g_1(c)= \displaystyle\sum_\Gamma \Gamma\int_c f_1(\Gamma),$ where $\Gamma$ runs through all the graphs in $\BVGra$. This sum is finite because the integral is zero every time the degree of $\Gamma$ differs from the degree of the chain $c$.   

Recall section \ref{BVKGra} where we saw that given a Cyclic Swiss Cheese type operad $\cP$ one can endow the total space $\prod_n \cP^2(\cdot, n)[n]$ with a a $\CPT-\cP^1$-bimodule structure. Moreover, morphism of Cyclic Swiss Cheese type operads induce morphisms of bimodules. Therefore we obtain a bimodule map

\begin{equation*}
\begin{tikzcd}[column sep=0.5em]				
\CPT \arrow{d}{\id} &\aol &\displaystyle\prod_n Chains(\mathsf F \mathbb H_{\bullet, n}) [-n]\arrow{d}{} & \aor &Chains(\FFM_2)\arrow{d}{} \\
\CPT & \aol  & \displaystyle\prod_n\BVKGra(\cdot, n)[-n]  &\aor & \BVGra .\\
\end{tikzcd}
\end{equation*}

We choose a Maurer Cartan element $\mu\in \left(\prod_nChains(\mathsf F \mathbb H_{0, n}) [-n]\right)_2$ to be 
$\theta = \prod_{n\geq 2} c_n$, where $c_n$ is the fundamental chain of the space $\mathsf F \mathbb H_{0,n}$. 

It is easy to see that the image of $c_n$ is zero for $n>2$ and for $n=2$ is the single graph in $\BVKGra(0,2)[-2]$ with no edges.

By twisting both $\prod_nChains(\mathsf F \mathbb H_{\bullet, n}) [-n]$ and $\prod_n\BVKGra(\cdot, n)[-n]$ with respect to $\mu$ and its image, we get a map of $Tw \CPT$-modules $\prod_nChains^{\mu}(\mathsf F \mathbb H_{\bullet, n}) [-n]\to\prod_n\BVKGra^{\mu}(\cdot, n)[-n]$ where the superscript $\mu$ indicates that there is a changed differential induced by the Maurer-Cartan elements. Since the ideal generated by \eqref{Ideal} acts as zero, we can restrict our action to the subquotient $\CBr$, of $Tw \CPT$, thus obtaining a morphism of left $\CBr$-modules.

Since the right action of $Chains(\FFM_2)$ on $Chains(\mathsf F \mathbb H)$ is on the non-boundary points, and analogously, the action of $\BVGra$ on $\BVKGra$ is on the type II vertices, it is clear that the morphism commutes with the right action. We obtain then the following bimodule map:

\begin{equation}\label{bimodmap:Chains to Gra}
\begin{tikzcd}[column sep=0.5em]				
\CBr \arrow{d}{} &\aol &  \displaystyle\prod_n Chains^{\mu}(\mathsf F \mathbb H_{\bullet, n}) [-n] \arrow{d}{} & \aor &Chains(\FFM_2)\arrow{d}{} \\
\CBr & \aol  & \displaystyle\prod_n\BVKGra^{\mu}(\cdot, n)[-n]  &\aor & \BVGra .\\
\end{tikzcd}
\end{equation}

The projection map $p_{m,n} \colon \mathsf F \mathbb H_{m, n}\to \mathsf F \mathbb H_{m, 0}$ that forgets the points at the boundary  induces a  strongly continuous chain \cite{semi-algebraic manifolds} $p_{m,n}^{-1} \colon \mathsf F \mathbb H_{m, 0} \to Chains(\mathsf F \mathbb H_{m, n})$.
Intuitively the image of a configuration of points in $\mathsf F \mathbb H_{m, 0}$ is the same configuration of points but with $n$ points at the real line that are freely allowed to move. If we consider the complex $Chains(\mathsf F \mathbb H_{\bullet, 0}) = \bigoplus_{m\geq 1}Chains(\mathsf F \mathbb H_{m, 0})$, this induces a degree preserving map 
$$p^{-1} \colon Chains(\mathsf F \mathbb H_{\bullet, 0}) \to \displaystyle\prod_{n\geq 0} Chains^{\mu}(\mathsf F \mathbb H_{\bullet, n}) [-n].$$

\begin{lemma}
$p^{-1}$ is a morphism of right $Chains(\FFM_2)$-modules and its image is a  $\CBr -Chains(\FFM_2)$-subbimodule.
\end{lemma}
\begin{proof}
The morphism clearly commutes with the right action. Let us check that $p^{-1}$ commutes with the differentials.

Let $c\in Chains(\mathsf F \mathbb H_{m, 0})$. 

The boundary term $\partial p_{m,n}^{-1}(c)$ has two kind of components. When at least two points at the upper half plane get infinitely close, giving us the term $p_{m,n}^{-1}(\partial c)$, and when points at the real line get infinitely close, giving us $\pm p_{m,n}^{f\partial}(c)$, where the $f\partial$ superscript represents that we are considering the boundary at every fiber.

Then, we have
$p^{-1}(\partial c) = \prod_{n\geq 0} p_{m,n}^{-1}(\partial c) =  \prod_{n\geq 0} \partial p_{m,n}^{-1}(c)  \pm p_{m,n}^{f\partial}(c)$. The first summand corresponds to the normal differential in $Chains(\mathsf F \mathbb H_{m, n})$ and the second summand is precisely the extra piece of the differential induced by the twisting.

It remains to check the stability under the left $\CBr$ action. It is enough to check the stability under the action of the generators $T_n^i, T_n^{i,i+1}, T'_n$ and $T'^{i,i+1}_n$.

Let $c_1, \dots, c_n\in Chains(\mathsf F \mathbb H_{\bullet, 0})$ of arbitrary degree. It is not hard to see that $$p^{-1}\circ p\left(T_n^1(p^{-1}(c_1),\dots,p^{-1}(c_n))\right) = T_n^1(p^{-1}(c_1),\dots,p^{-1}(c_n)).$$ This follows essentially from the fact that on the right hand side the projection in $Chains^{\mu}(\mathsf F \mathbb H_{\bullet, k}) [-k]$ is the sum over all the possibilities of distributing $k_i$ points on the boundary stratum of $c_i$, for $i=2,\dots,n$ and $k_1$ boundary points not infinitely close to any of these chains, with $k_1+...+k_n=k$, whereas the left hand is taking all of these possibilities into account at once.

For the remaining $T_n^{i}$, the stability follows from the remark that if a chain is in the image of $p^{-1}$, then any cyclic permutation of it is still in the image of $p^{-1}$. Since forgetting one of the boundary points of a chain in the image of $p^{-1}$ leaves it in the image of $p^{-1}$, we get stability under the action of $T_n^{j,j+1}$. 

The other generators follow from similar arguments.
\end{proof}

$p^{-1}$ is right inverse to the projection map, therefore it is an embedding of right $Chains(\FFM_2)$-modules. We can therefore transport back the left $\CBr$ action on its image, making $p^{-1}$ a morphism of $\CBr-Chains(\FFM_2)$-bimodules.

By composition with the map \eqref{bimodmap:Chains to Gra}, we obtain the following bimodule map:

\begin{equation}\label{Chains to KGra}
\begin{tikzcd}[column sep=0.5em]				
\CBr \arrow{d}{} &\aol &   Chains(\mathsf F \mathbb H_{\bullet, 0}) \arrow{d}{} & \aor &Chains(\FFM_2)\arrow{d}{} \\
\CBr & \aol  & \displaystyle\prod_n\BVKGra^{\mu}(\cdot, n)[-n]  &\aor & \BVGra .\\
\end{tikzcd}
\end{equation}

\subsection{A representation on the colored vector space $\Dpoly\oplus\Tpoly$}\label{section:representation on T+D}

In this section we drop the $\mathbb R^d$ from the notation $\Tpoly$, $\TDpoly$ and $\Dpoly$, for simplicity. In Section \ref{section:globalization} we globalize the results obtained here.

Let $x_1,\dots,x_n$ be coordinates in $\mathbb R^n$ and let $\xi_1, \dots, \xi_n$ be the corresponding basis of vector fields.
We define an action of $\BVGra$ on the graded algebra of multivector fields $\Tpoly$ in $\mathbb R^d$ by setting
$$\Gamma(X_1,\dots, X_k) = \left(\prod_{(i,j)\in\Gamma} \sum_{l=1}^d \frac{\partial}{\partial x_l^{(j)}}\wedge \frac{\partial}{\partial \xi_l^{(i)}} \right) (X_1\wedge\dots\wedge X_k),$$
where $\Gamma\in \BVGra(k)$, $X_1,\dots,X_k$ are multivector fields, the product runs over all edges of $\Gamma$ in the order given by the numbering of edges and the superscripts $(i)$ and $(j)$ mean that the partial derivative is being taken on the $i$-th and $j$-th component of $X_1,\dots,X_k$. This is equivalent to an operad morphism $\BVGra \to \End \Tpoly$.

Seeing $\Gamma$ as an element of $\BVGra(m+n)$ and, using the action of $\BVGra$ in $\Tpoly$, together with the fact that $C^\infty$ funtions are degree zero multivector fields we define a map $g\colon\BVKGra(m, n) \to \Hom(\Tpoly^{\otimes m}\otimes C_c^\infty(\mathbb R^d)^{\otimes n},C_c^\infty(\mathbb R^d))$ by
\begin{equation}\label{composition1}
g(\Gamma)(X_1,\dots,X_m)(f_1,\dots,f_n) = \Gamma(X_1,\dots,X_m,f_1,\dots,f_n). \footnote{We set all $\xi_i=0$.}
\end{equation}

These two maps form a colored operad morphism from $\left(\BVGra, \BVKGra\right)$ to the Swiss Cheese type operad $\left(\End \Tpoly , \Hom(\Tpoly^{\otimes m}\otimes C_c^\infty(\mathbb R^d)^{\otimes n},C_c^\infty(\mathbb R^d))\right)$, a suboperad of the colored operad $\End \left(\Tpoly \oplus C^\infty_c(\mathbb R^d)\right)$.

The Tensor-Hom adjunction allows us to rewrite $\Hom(\Tpoly^{\otimes m}\otimes C_c^\infty(\mathbb R^d)^{\otimes n},C_c^\infty(\mathbb R^d))$ as $\Hom\left(\Tpoly^{\otimes m}, \Hom\left( C_c^\infty(\mathbb R^d)^{\otimes n},C_c^\infty(\mathbb R^d)\right)\right)$ and the bilinear form $\int \colon C_c^\infty(\mathbb R^d) \otimes C_c^\infty(\mathbb R^d) \to \mathbb R$ induces a map 

\begin{equation}\label{composition2}
\Hom\left(\Tpoly^{\otimes m}, \Hom\left( C_c^\infty(\mathbb R^d)^{\otimes n},C_c^\infty(\mathbb R^d)\right)\right)\to 
\Hom\left(\Tpoly^{\otimes m}, \Hom\left( C_c^\infty(\mathbb R^d)^{\otimes n+1},\mathbb R\right)\right).
\end{equation}

There is a natural $C_{n+1}$ action on $\Hom\left(\Tpoly^{\otimes m},\Hom\left( C_c^\infty(\mathbb R^d)^{\otimes n+1},\mathbb R\right)\right)$ given by the action on $C_c^\infty(\mathbb R^d)^{\otimes n+1}$ and also a distinguished element $\mathbbm 1$ map given by the insertion of the constant function $\equiv 1$ on the first input of $\Hom\left( C_c^\infty(\mathbb R^d)^{\otimes n+1},\mathbb R\right)$.

\begin{lemma}
With the above described map and cyclic action, the composition of the maps \eqref{composition1} and \eqref{composition2} induces a morphism of Cyclic Swiss Cheese type operads

$$ \left(\BVGra,\BVKGra\right) \to \left( \End \Tpoly, \Hom\left(\Tpoly^{\otimes \bullet}, \Hom( C_c^\infty(\mathbb R^d)^{\otimes \bullet +1},\mathbb R)\right)\right). $$
\end{lemma}

\begin{proof}
It is clear that the map is a morphism of colored operads and it sends one distinguished element to the other. It is enough to check the compatibility with the cyclic action.

Notice that the image of a graph under the morphism $$\BVKGra(m,n) \to \Hom\left(\Tpoly^{\otimes m}, \Hom( C_c^\infty(\mathbb R^d)^{\otimes n+1},\mathbb R)\right)$$ actually lands inside of $\Hom\left(\Tpoly^{\otimes m}, {\TDpoly}(n)\right)$ and this space is an algebra with product given by the product of functions.

It is clear by the definition of this morphism that it commutes with products, therefore to check the compatibility with the cyclic action it is enough to check it on graphs with just one edge.

Let $\Gamma^i_j \in \BVKGra(m,n)$. Recall that the action of the generator $\sigma$ of $C_{n+1}$ on $\Gamma^i_j$ is $\sigma(\Gamma^i_j)= \Gamma^i_{j-1}$ if $j\ne 1$ and 
 $\sigma(\Gamma^i_1) = -\sum_{k=1}^n \Gamma^i_k - \sum_{k=1}^m \Gamma^{i,k}$.
 The action of $\sigma$ on $\Gamma^{i,j}\in\BVKGra(m,n)$ is $\sigma(\Gamma^{i,j})= \Gamma^{i,j}$, for $1\leq i,j\leq m$.

Let $X_1,\dots,X_m\in \Tpoly$ and let $f_0,\dots,f_{n}\in C^\infty(\mathbb R^d)$.

Notice that $g(\Gamma^i_1)(X_1,\dots,X_m)$ can only be non-zero if all the $X_j$, for $j\ne i$ are in ${\Tpoly}^0 = C^\infty(\mathbb R^d)$ and $X_i \in {\Tpoly}^1 = \Gamma(\mathbb R^d, T_{\mathbb R^d})$.

 The operator $\left( g(\Gamma^i_1)(X_1,\dots,X_m)\right)^{\sigma}$ is defined by
$$\int f_0g(\Gamma^i_1)(X_1,\dots,X_m)(f_1,\dots,f_{n}) = \int f_1\left( g(\Gamma^i_1)(X_1,\dots,X_m)\right)^{\sigma}(f_2\dots,f_{n},f_0),$$ i.e., by ``taking the derivatives from $f_1$".

Let us write $X_i = \sum_{k=1}^d \psi_k \frac{\partial }{\partial x_k}$. Expanding the first integral we have 

\begin{flalign*}
&\int f_0g(\Gamma^i_1)(X_1,\dots,X_m)(f_1,\dots,f_{n})=\\
&\sum_{k=1}^d\int \frac{\partial f_1}{\partial x_k}\psi_k X_1\dots\hat{X_i}\dots X_mf_2\dots f_{n}f_0=\\
&-\sum_{k=1}^d\int  f_1\frac{\partial \psi_k}{\partial x_k}X_1 \dots \hat{X_i}\dots X_m f_0 f_2 \dots f_{n}+f_1\psi_k\frac{\partial X_1}{\partial x_k}X_2\dots \hat{X_i}\dots X_mf_2\dots f_{n}f_0 + \\ 
&+\dots + f_1\psi_kX_1\dots \hat{X_i},\dots X_mf_2\dots f_n\frac{\partial f_0}{\partial x_k}.
\end{flalign*}

Therefore 

\begin{flalign*}
& \left( g(\Gamma^i_1)(X_1,\dots,X_m)\right)^{\sigma}(a_1,\dots,a_n) =\\
& - \Gamma^{i,i}(X_1,\dots,X_m,a_1,\dots,a_n)- \sum_{k=1,k\ne i}^m \Gamma^{i,k}(X_1,\dots,X_m,a_1,\dots,a_n) -\sum_{k=1}^n \Gamma^i_k(X_1,\dots,X_m,a_1,\dots,a_n)=\\
&g( - \sum_{k=1}^m \Gamma^{i,k}-\sum_{k=1}^n \Gamma^i_k)(X_1,\dots,X_m)(a_1,\dots,a_n) =\\
&g( \Gamma^i_1 \cdot \sigma) (X_1,\dots,X_m)(a_1,\dots,a_n) .
\end{flalign*} 

The verification for the case $\Gamma^{i,j}$ is trivial and the case $\Gamma^i_j$ with $j\ne 1$ is also immediate because there is only permutation of variables involved.
\end{proof}

We obtain then a bimodule map

\begin{equation}\label{KGra to End}
\begin{tikzcd}[column sep=0.5em]				
\CPT \arrow{d}{\id} &\aol &\displaystyle\prod_n\BVKGra(\cdot, n)[-n] \arrow{d}{} & \aor &\BVGra\arrow{d}{} \\
\CPT & \aol  & \displaystyle\prod_n \Hom\left(\Tpoly^{\otimes \bullet}, \Hom( C_c^\infty(\mathbb R^d)^{\otimes n+1},\mathbb R)\right)[-n]  &\aor & \End \Tpoly .\\
\end{tikzcd}
\end{equation}

The image of the Maurer-Cartan element 
\begin{tikzpicture}
\node[label=below :$\overline{1}$]  at (0.5,0) [int] {};
\node[label= below : $\overline{2}$]  at (1.5,0) [int] {};
\draw (0,0)--(2,0);
\end{tikzpicture}$\in \BVKGra(0,2)[-2]$ is the element induced by the multiplication map $\mu \colon C_c^\infty(\mathbb R^d)^{\otimes 2} \to C_c^\infty(\mathbb R^d)$. 

By twisting with respect to these Maurer-Cartan elements we obtain a map of $Tw \CPT$ from $\displaystyle\prod_n\BVKGra^{\mu}(\cdot, n)[-n]$  to $ \Hom^{\mu}\left(\Tpoly^{\otimes \bullet}, \displaystyle\prod_n \Hom( C_c^\infty(\mathbb R^d)^{\otimes n+1},\mathbb R)[-n]\right)$. Notice that in this last space, the differential coming from the twisting is the same as the one induced by the Hochschild differential and the degrees also agree with the Hochschild complex. In fact, the image of the map \eqref{KGra to End} lands in $\Hom\left(\Tpoly^{\otimes \bullet}, \Dpoly\right)$.

Since \begin{tikzpicture}[scale=1]
\node (v0) at (0,0) [int] {.};
\draw (v0) -- (0,0.6);
\draw (v0) -- (0,-0.6);
\draw[ultra thick] (v0) --  +(-45:0.4cm);
\end{tikzpicture}
\begin{tikzpicture}[scale=1]
\node at (0,0) {};
\node at (0,0.5) {$-$};
\end{tikzpicture}
\begin{tikzpicture}[scale=1]

\draw (0,0.6) -- (0,-0.6);
\end{tikzpicture}$\in Tw \CPT$ acts trivially on both spaces, this induces an action of its subquotient $\CBr$, therefore we obtain the following maps of bimodules:

\begin{equation}\label{bimodmap:Gra to Hom}
\begin{tikzcd}[column sep=0.5em]				
\CBr \arrow{d}& \aol  & \displaystyle\prod_n\BVKGra^{\mu}(\cdot,n)[-n] \arrow{d} &\aor & \BVGra \arrow{d}\\
\CBr &  \aol  &  \Hom (\Tpoly^{\otimes \bullet},\Dpoly)      &  \aor   & \End \Tpoly.\\
\end{tikzcd}
\end{equation} 

Also, the $\CBr$ action on $\Hom(\Tpoly^{\otimes \bullet},\Dpoly)$ comes from the action of $\CBr$ on $\Dpoly$ (as seen in \ref{Dpoly}), which translates into an operadic morphism $\CBr\to \End \Dpoly$.
Thus, by composition with the map \eqref{Chains to KGra} we obtain

\begin{equation*}
\begin{tikzcd}[column sep=0.5em]				
\CBr \arrow{d}{} &\aol & Chains(\mathsf F \mathbb H_{\bullet, 0}) \arrow{d}{} & \aor &Chains(\FFM_2)\arrow{d}{} \\
\End \Dpoly &  \aol  &  \Hom(\Tpoly^{\otimes \bullet},\Dpoly)      &  \aor   & \End \Tpoly.\\
\end{tikzcd}
\end{equation*}
\subsection{A zig-zag of quasi-torsors}
Let us recall the definition of an \textit{operadic quasi-torsor} from \cite{operadic torsors}: 

\begin{defi}
Let $\mathcal P$ and $\mathcal Q$ be two differential graded operads and let $M$ be a $\mathcal{P}-\mathcal Q$ operadic differential graded bimodule, i.e., there are compatible actions $$\mathcal P \aol M \aor \mathcal Q.$$
We say that $\cM$ is a $\mathcal P$-$\mathcal Q$ quasi-torsor if there is an element $\mathbf 1 \in M^0(1)$ such that the canonical maps
\begin{equation}\label{modmaps}
\begin{aligned}
 l\colon \cP &\to \cM  \quad\quad\quad\quad\quad & r\colon\cQ & \to \cM \\
 p   &\mapsto p\circ (\mathbf 1,\dots, \mathbf 1)  & q &\mapsto \mathbf 1\circ q \\
\end{aligned}
\end{equation}
are quasi-isomorphisms.
\end{defi}

\begin{lemma}\label{lemma:quasi-torsor}
$Chains(\mathsf F \mathbb H_{\bullet, 0})$ is a $\CBr-Chains(\FFM_2)$ quasi-torsor.
\end{lemma}

\begin{proof}

Let us consider the element $\mathbf 1\in Chains_0(\mathsf F \mathbb H_{1, 0})$ corresponding to a single point on the upper half plane with frame is pointing upwards.

Let $i\colon \FFM_2 \to \mathsf F \mathbb H_{\bullet, 0}$ be the map that sends a configuration in $c\in\FFM_2$ to the configuration in $\mathsf F \mathbb H_{\bullet, 0}$ given by one boundary stratum on the upper half plane with $c$ on it. It is clear that $i$ is a homotopy equivalence (with homotopy inverse being the map that ``forgets'' the boundary of the upper half plane).
The map $r\colon Chains(\FFM_2) \to Chains(\mathsf F \mathbb H_{\bullet, 0})$, as in Definition \ref{modmaps} is the image of $i$ via the functor $Chains$. Since $i$ is a homotopy equivalence, $r$ is a quasi-isomorphism.

It was shown in \cite{framed homology} that $H(\FFM_2)= \BV$.

The map $l$ sends \begin{tikzpicture}[scale=1]
\node (v0) at (0,0) [ext] {1};
\node (ro) at (0,0.6) {*};

\draw (v0) -- (0,0.6);
\draw[ultra thick] (v0) --  +(-45:0.5cm);
\end{tikzpicture}$\in \CBr_{-1}(1)$ to the fundamental chain of the circle. It sends  \begin{tikzpicture}[scale=0.8]

\node (ro) at (1,1.7) {$*$};
\node (e0) at (1,1) [int] {$*$};
\node (e1) at (0.5,0) [ext] {$1$};
\node (e2) at (1.5,0) [ext] {$2$};
\draw (e0) -- +(0,.7);
\draw[-] (e1) edge (e0) (e2) edge (e0);


\draw[ultra thick] (e1) --  ($(e1)!0.5cm!(e0)$);
\draw[ultra thick] (e2) --  ($(e2)!0.5cm!(e0)$);
\end{tikzpicture} to the zero chain consisting of two horizontally aligned points in the upper half plane with frames pointing upwards.
And it sends \begin{tikzpicture}[scale=1]
\node (v0) at (5,6.7) [ext] {$2$};
\node (v1) at (5,7.5) [ext] {$1$};
\node (v3) at (5,8.1)  {$*$};
\node (v10) at (6.2,7.5) [ext] {2};
\node (v9) at (6.2,6.7) [ext] {1};
\node (v12) at (6.2,8.1)  {*};
\node (v14) at (5.6,7.1)  {+};

\draw (v1) -- +(0,.6);
\draw (v10) -- +(0,.6);

\draw (v0) -- (v1);
\draw (v9) -- (v10);

\draw[ultra thick] (v0) --  ($(v0)!0.4cm!(v1)$);
\draw[ultra thick] (v1) --  ($(v1)!0.4cm!(v3)$);

\draw[ultra thick] (v10) --  ($(v10)!0.4cm!(v12)$);
\draw[ultra thick] (v9) --  ($(v9)!0.4cm!(v10)$);
\end{tikzpicture} to the 1-chain corresponding to two points rotating around each other.

Since the homologies of $\CBr$ and of $\mathsf F \mathbb H_{\bullet, 0}$ are both $\BV$ and $l$ sends (representatives of) generators to (representatives of) generators, $l$ is a quasi-isomorphism.
\end{proof}

The main Theorem of \cite{operadic torsors} states that if the $\cP-\cQ$-bimodule $M$ is an operadic quasi-torsor, then there is a zig-zag of quasi-isomorphisms connecting $\mathcal P \aol M \aor \mathcal Q$ to the canonical bimodule $\mathcal P \aol \cP \aor \mathcal P$.

It follows then from Lemma \ref{lemma:quasi-torsor} that there is a zig-zag of bimodules 

$$
\begin{tikzcd}[column sep=0.5em]
\CBr \arrow{d}{} & \aol& \CBr\arrow{d}{}& \aor& \CBr\arrow{d}{} \\
\cdots & \aol& \cdots & \aor& \cdots \\
\CBr \arrow{u}{}& \aol  & Chains(\mathsf F \mathbb H_{\bullet, 0}) \arrow{u}{}  &\aor & Chains(\FFM_2) \arrow{u}{} .
\end{tikzcd} $$

Let $\CBr_\infty^{\text{bimod}}$ be a cofibrant resolution of the canonical bimodule $\CBr$. $\CBr_\infty^{\text{bimod}}$ is a $\CBr_\infty-\CBr_\infty$-bimodule, where $\CBr_\infty$ is a cofibrant resolution of the operad $\CBr$.

Finally, the zig-zag can be lifted up to homotopy to a bimodule map

$$\begin{tikzcd}[column sep=0.5em]				
\CBr_\infty \arrow{d}{} &\aol &\CBr_\infty^{\text{bimod}} \arrow{d}{} & \aor &\CBr_\infty \arrow{d} \\
\End \Dpoly &  \aol  &  \Hom(\Tpoly^{\otimes \bullet},\Dpoly)      &  \aor   & \End \Tpoly.\\
\end{tikzcd}$$

giving us the desired quasi-isomorphism and thus proving Theorem \ref{main theorem}. 

It also follows from Lemma \ref{lemma:quasi-torsor} and \cite{operadic torsors} that $\CBr$ is quasi-isomorphic to $Chains(\FFM_2)$. Due to the formality of $\FFM_2$ \cite{formality framed discs}, it follows that we can replace $\CBr_\infty$ in Theorem \ref{main theorem} by any cofibrant replacement of the operad $\BV$.

\section{Globalization}\label{section:globalization}
Let $M$ be a $d$-dimensional oriented manifold.
In this section we globalize the $\BV_\infty$ quasi-isomorphism $\Tpoly(\mathbb R^d) \to \Dpoly(\mathbb R^d)$ from Theorem \ref{main theorem} to a quasi-isomorphism $\Tpoly(M) \to \Dpoly(M)$, thus proving Theorem \ref{main global}. To do this we use standard formal geometry techniques.

\subsection{The idea:}\label{subsection:idea}
We refer the reader to the paper \cite{Dolgushev}, from which we borrow the notation.
 
 Theorem \ref{main theorem} is valid if we replace $\mathbb R^d$ by $\Rformal$, its formal completion at the origin, i.e., the space whose ring of functions is given by formal power series on the coordinates $x_1,\dots,x_d$.  \\
 
 We  consider $\mathcal T_\text{poly}$ (resp. $\mathcal D_\text{poly}$),  the vector bundle on $M$ of fiberwise formal multivector fields (resp. multidifferential operators) tangent to the fibers. We can then construct the vector bundles $\Omega(\mathcal T_\text{poly}, M)$ of forms valued in $\mathcal T_\text{poly}$ and $\Omega(\mathcal D_\text{poly}, M)$ of forms valued in $\mathcal D_\text{poly}$ with appropriate differentials. 
 
The fibers of the bundles $\mathcal T_\text{poly}$ and $\mathcal D_\text{poly}$ are isomorphic to $\Tpoly(\Rformal)$ and $\Dpoly(\Rformal)$, respectively. Therefore, the formal version of the formality map can be used to find a vector bundle $\CBr_\infty$ quasi-isomorphism

\begin{equation}\label{fiberwise morphism}
  U^f\colon \Omega(\mathcal T_\text{poly}, M) \to \Omega(\mathcal D_\text{poly}, M). \footnote{Using the fact that the formality morphism is invariant by linear transformation of coordinates.}
\end{equation}

These two vector bundles can be related with $\Tpoly(M)$ and $\Dpoly(M)$. In fact,  with an appropriate change of differential that comes from a choice of a flat connection, $\Omega(\mathcal T_\text{poly}, M)$ becomes a resolution of $\Tpoly(M)$ and $\Omega(\mathcal D_\text{poly}, M)$ becomes a resolution of $\Dpoly(M)$. This change of differential can be seen locally as a twist via a Maurer-Cartan element $B$ in $\Omega^1(\mathcal T^1_\text{poly}, U) = \Omega^1(\mathcal D^1_\text{poly}, U)$. However, the linear part of $B$ (in the fiber coordinates) is not globally well defined.

\subsection{An extension of Kontsevich's $L_\infty$ morphism}\label{subsection:extension}
In this section we show that the $\BV_\infty$ formality morphism from Theorem \ref{main theorem} can be obtained in such a way that it extends Kontsevich's original $L_\infty$ morphism \cite{Kontsevich}.

We have the following chain of maps:

\begin{equation}\label{full diagram}
\begin{tikzcd}[column sep=0.5em]
\hoLie\arrow{d}{} &\aol & \hoLie^{\text{bimod}}  \arrow{d}{} & \aor &\hoLie  \arrow{d}{} \\
\CBr_\infty \arrow{d}{} &\aol & \CBr_\infty^{\text{bimod}}  \arrow{d}{} & \aor &\CBr_\infty  \arrow{d}{} \\
\CBr \arrow{d}{} & \aol& Chains(\mathsf F \mathbb H_{\bullet, 0})\arrow{d}{}& \aor& Chains(\FFM_2)\arrow{d}{} \\
\CBr \arrow{d}& \aol  & \displaystyle\prod_n\BVKGra^{\mu}(\cdot,n)[-n] \arrow{d} &\aor & \BVGra \arrow{d}\\
\End \Dpoly &  \aol  &  \Hom(\Tpoly^{\otimes \bullet},\Dpoly)      &  \aor   & \End \Tpoly.\\
\end{tikzcd} 
\end{equation}

where $\hoLie = \Omega( \Lie\{1\}^{\vee})$, the first downwards maps are induced by the inclusion $\Lie\to \CBr$ and the other maps follow from the proof of Theorem \ref{main theorem}. Showing that our morphism extends Kontsevich's formality morphism amounts to showing that the full composition of the maps in \eqref{full diagram} gives Kontsevich's map. This is clear for the left column. For the other two columns the argument is similar so we will only prove it for the right column given that the notation is simpler.  Let us call $\mu_n$ the generator of $\Lie\{1\}^{\vee}(n)$.


 We recall that in \cite{Kontsevich}
 the construction of $U_n$, the $L_\infty$ components of the formality morphism are constructed by sending $\mu_n$ to the fundamental chain of $\mathbb H_{n,0}$. We wish then to show the commutativity of the following diagram, where the uppers horizontal maps represent Kontsevich's approach and $\Gra$ is the suboperad of $\BVGra$ in which tadpoles are not admitted.

\begin{equation}\label{extension}
\begin{tikzcd}
\hoLie \arrow{r}\arrow{d} & Chains(\FM_2)\arrow{r} & \Gra \arrow{r}& \End(\Tpoly).\\
\CBr_\infty \arrow{r} & Chains(\FFM_2) \arrow{r} & \BVGra \arrow{ru}
\end{tikzcd}
\end{equation}

As semi-algebraic manifolds, $\FFM_2(n) = \FM_2(n) \times (S^1)^{\times n}$, therefore there exists an inclusion map $i \colon \FM_2 \to \FFM_2$ that is the identity on the $\FM_2$ component and constant equal to the vertical direction in the $S^1$ components.

Naming the relevant maps, diagram \eqref{extension} becomes

\begin{equation}
\begin{tikzcd}
\hoLie \arrow{r}{f}\arrow{d}{i_L} & Chains(\FM_2)\arrow{r}\arrow{d}{i_*} & \Gra \arrow{r}\arrow[hook]{d}& \End(\Tpoly).\\
\CBr_\infty \arrow{r}{g} & Chains(\FFM_2) \arrow{r} & \BVGra \arrow{ru}
\end{tikzcd}
\end{equation}

It is clear that the right triangle diagram and the adjacent square diagram are commutative. To conclude the commutativity of the exterior diagram it is enough to show that the left square is commutative but this need not be the case. Fortunately this can be rectified if one is careful when constructing the map $g$ as a lift over quasi-isomorphisms. We sketch here the argument that is nothing but an adapted version or the argument of Lemmas 12 and 13 in \cite{operadic torsors}.

The fact that $\Lie\{1\}$ can be seen embedded in $\CBr$ via the map $F$ in section \ref{Def of CBr} implies that the generators $\mu_n$ of $\hoLie$ can be seen as part of the generators of $\CBr_\infty$ (via the map $i_L$) and
The map $f$ sends $\mu_n$ to the fundamental chain of $\FM_2(n))$. 

To construct $g$ one starts with a filtration $0=\mathcal F^0\subset \mathcal F^1\subset \dots \subset\CBr_\infty$ such that when differentiating the generators we fall in the previous degree of the filtration and then we construct the map recursively using the following diagram:

\begin{tikzcd}
\ & \ & F\arrow[twoheadrightarrow]{d}\arrow[twoheadrightarrow]{dr} & \\
\CBr_\infty \arrow[dotted]{rru}{g'}\arrow{r} & \CBr \arrow{r} & E & Chains(\FFM_2)\arrow{l}\\
\end{tikzcd}

where all maps are quasi-isomorphisms, $E$ is the operad through which the zig-zag connecting $\CBr$ and $Chains(\FFM_2)$ goes and $F$ is the operad resulting from the ``surjective trick'', i.e., an operad that surjects both onto $E$ and $Chains(\FFM_2)$ such that the depicted triangle commutes up to homotopy. At every stage we wish to map $\mu_n$ to a pre-image of the fundamental chain of $\FM_2$ (seen inside of $\FFM_2$) and essentially one has to check that $dg'(\mu_n) = g'(d\mu_n)$, but this follows from the fact that the boundary of the fundamental chain of $\FM_2(n)$ is computed the same way as the cocomposition of $\mu_n$ in $\Lie\{1\}^{\vee}$.

\subsection{The bimodule BVKGraphs}
Notice that due to the chain of morphisms \eqref{full diagram} there is a morphism of bimodules 

\begin{equation*}
\begin{tikzcd}[column sep=0.5em]				
\hoLie\arrow{d}{} &\aol & \hoLie^{\text{bimod}}  \arrow{d}{} & \aor &\hoLie  \arrow{d}{} \\
\CBr & \aol  &  \displaystyle\prod_n\BVKGra^\mu(\cdot, n)[-n] &\aor & \BVGra.\\
\end{tikzcd}
\end{equation*}

Using the formalism of twisting of bimodules described in the Appendix we can perform the bimodule twisting with respect to this morphism, thus obtaining the operadic bimodule $Tw\CBr  \aol   Tw \displaystyle\prod_n\BVKGra^\mu(\cdot, n)[-n] \aor Tw \BVGra$ .

The elements in $Tw \BVGra (n)$ can be seen as linear combinations of directed graphs with at least $n$ vertices, where from these, $n$ of them are labeled by numbers from $1$ to $n$  and the remaining ones are indistinguishable. The labeled vertices are called external vertices and the unlabeled ones are called internal vertices.
In a similar way, the elements of $ Tw \displaystyle\prod_n\BVKGra^\mu(m, n)[-n]$ consist of the same kind of graphs, but where now the type I vertices come in two flavors, the indistinguishable internal vertices and the $m$ labeled external vertices.

\begin{prop-def}\label{prop-def:BVGraphs}
The operad $Tw \BVGra$ has a suboperad that we call $\BVGraphs$ spanned by graphs satisfying the following properties:

\begin{enumerate}
  \item There are no 1-valent internal vertices or 2-valent internal vertices with exactly one incoming and one outcoming edges;
	  \item There are no tadpoles on internal vertices.
\end{enumerate}
\end{prop-def}
\begin{proof}
It is clear that the operadic composition preserves each of the conditions imposed in $\BVGraphs$, therefore we only need to check that $\BVGraphs$ is preserved under the action of the differential.

The differential $d$ in $Tw \BVGra$ has the form $d=d_1+d_2$, where $d_1$ is defined by 
$d_1\Gamma = ($
\begin{tikzpicture}
\node (e1) at (0,0) [ext] {};
\node (e2) at (0.5,0) [int] {};

\draw [->] (e1)--(e2);
\end{tikzpicture}$+$
\begin{tikzpicture}
\node (e1) at (0,0) [ext] {};
\node (e2) at (0.5,0) [int] {};

\draw [->] (e2)--(e1);
\end{tikzpicture}
$)\circ \Gamma$
-$ \sum_i \pm\Gamma\circ_i($
\begin{tikzpicture}
\node (e1) at (0,0) [ext] {};
\node (e2) at (0.5,0) [int] {};

\draw [->] (e1)--(e2);
\end{tikzpicture}
$+$ 
\begin{tikzpicture}
\node (e1) at (0,0) [ext] {};
\node (e2) at (0.5,0) [int] {};

\draw [->] (e2)--(e1);
\end{tikzpicture}$)$, and $d_2$ acts by replacing every internal vertex by \begin{tikzpicture}
\node (e1) at (0,0) [int] {};
\node (e2) at (0.5,0) [int] {};

\draw [->] (e1)--(e2);
\end{tikzpicture}.

 This means that the differential acts by splitting internal vertices out of every vertex.

Splitting any vertex cannot create create tadpoles at a vertex, therefore property $(2)$ is preserved by the differential.

The $d_2$ component of the differential produces $1-$valent internal vertices when all incident edges are reconnected to only one of the internal vertices. Similarly, the second summand in $d_1$ produces a $1-$valent internal vertex whenever all incident edges are reconnected to the external vertex. All of these factors are canceled out by the first summand of the definition of $d_1$. 

The creation of internal vertices with exactly one incoming and one outcoming edges happens only when after taking the differential in one vertex, there is exactly one other vertex that connects to the split internal vertex. However this term will be canceled out when the differential is taken on this other vertex.
\end{proof}

Given an operad $\cP$ with a morphism from $\hoLie\to \cP$, there is a canonical projection $Tw \cP \to \cP$, as described in the Appendix.
We prove now a Lemma that will be useful to show that the operad morphism $Chains(\FFM_2) \to \BVGra$ factors through $\BVGraphs$.

\begin{lemma}\label{lemma: chains natively twistable}
$Chains(\FFM_2)$ is natively twistable.
\end{lemma}

\begin{proof}
We need to construct an operad map $\iota\colon Chains(\FFM_2) \to Tw \ Chains(\FFM_2)$ that is a right inverse to the canonical projection.

Let $\FFM_2^k(n+k)$ be the subspace of $\FFM_2(n+k)$ whose last $k$ points have their frame constantly pointing upwards.

The bundle maps $\pi_{n,k}\colon \FFM_2^k(n+k)\to \FFM_2(n)$ defined by ``forgetting" the last $k$ points define a map at the level of chains 
$$\pi^{-1}_{k,n} \colon Chains(\FFM_2(n)) \to Chains(\FFM_2^k(n+k)) \subset Chains(\FFM_2(n+k)).$$ 

Notice that these map lands in the $\mathbb S_k$ invariant subspace $Chains(\FFM_2(n+k))^{\mathbb S_k}$.

Let $c\in Chains(\FFM_2)(n)$. To define $\iota(c)$ it is enough to define its projection in  $Chains(\FFM_2)(n+k)^{\mathbb S_k}$. We define this projection to be $\pi^{-1}_{k,n}$.

To see that this is an operad map, we need to check that $\iota(c\circ_i c') = \iota (c) \circ_i \iota(c')$. This equality follows from the observation that fixed a boundary stratum of a configuration of points, having $k$ points varying freely is the same as $i$ points inside that boundary stratum and $k-i$ outside, for $i=0,\dots,k$.
\end{proof}

 The operad morphism $Chains(\FFM_2)\to \BVGra$ and the functoriality of $Tw$ and the canonical projections $Tw \cP \to \cP$ give us the following commutative square

$$\begin{tikzcd}
Tw \ Chains(\FFM_2) \arrow{d} \arrow{r} & Tw\BVGra \arrow{d}\\
Chains(\FFM_2) \arrow{r}& \BVGra
\end{tikzcd} 
$$

As a corollary of the previous Lemma, the operad morphism $Chains(\FFM_2) \to \BVGra$ factors as $Chains(\FFM_2) \to Tw \BVGra\to \BVGra$. Explicitly, the first map is given by

\begin{equation}\label{map:chains to bvgraphs}
c \in    Chains(\FFM_2)(n) \mapsto \sum_\Gamma \Gamma \int_{\pi_{\Gamma}^{-1}(c)}f_1(\Gamma),
\end{equation}
where $f_1(\Gamma)$ is the form associated to the graph $\Gamma$, as defined in Section \ref{subsec: chains to graphs} and for $\Gamma$ a graph with $n$ external and $m$ internal vertices, $\pi^{-1}_{\Gamma}(c)$ is the chain in $Chains(\FFM_2)(n+m)$ in which the $m$ points corresponding to the internal vertices vary freely in $\mathbb R^d$ while their frame is constantly pointing upwards.

\begin{proposition}\label{lemma:factors through BVGraphs}
The operad morphism $Chains(\FFM_2) \to \BVGra$ defined above factors through $\BVGraphs$.
\end{proposition}
\begin{proof}
It is enough to check that the morphism \eqref{map:chains to bvgraphs} lands in $\BVGraphs$ and for this one must check that the coefficient of the graphs that are ``forbidden'' in $\BVGraphs$ is zero. This is clear if the graph contains a 1-valent internal vertex, since the computation of the coefficient involves an integral of a 1-form (corresponding to the incident edge) over a 2 dimensional space.\\

Suppose the graph $\Gamma$ contains an internal vertex with exactly one incoming and one outcoming edges. Let us call this vertex $i$ and let us also call $a$ and $b$ the vertices to which these two edges connect.
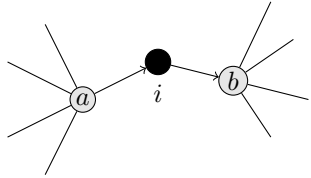
\begin{figure}[h]

\begin{tikzpicture}

\node[label=below:$i$] (int) at (0,0) [int]{$a$};

\node[fill=gray!20] (a) at (-1,-0.5) [ext] {$a$};

\node[fill=gray!20] (b) at (1,-0.25) [ext] {$b$};

\draw [->] (a)--(int);
\draw [->](int)--(b);

\draw (a) -- (-1.5,0.5);
\draw (a) -- (-2,0);
\draw (a) -- (-2,-1);
\draw (a) -- (-1.5,-1.5);
\draw (b) -- (1.5,0.8);
\draw (b) -- (1.8,0.3);
\draw (b) -- (2,-0.5);
\draw (b) -- (1.5,-1);
\end{tikzpicture}
\caption{An internal vertex connected to two (internal or external) vertices.}
\end{figure}

 By Fubini's Theorem for fibrations, the integral $\int_{\pi_{\Gamma}^{-1}(c)}f_1(\Gamma)$ can be rewritten as
 
 $$\int \left(\int_{X_{z_a,z_b}} d\phi_{ai}d\phi_{ib} \right)\dots,  $$ 
 where $X_{z_a,z_b}$ is the space of configurations in which the points labeled by $a$ and $b$ are in positions $z_a$ and $z_b$, and the point labeled by $i$ moves freely. It suffices therefore to show that the integral 
\begin{equation}\label{integral} 
 \int_{ X_{z_a,z_b}} d\phi_{ai}d\phi_{ib}
\end{equation} 
  vanishes. To check this, notice that by (the fibration integral version of) Stokes Theorem, we have
 
$$d \underbrace{\int_{Y_{z_a,z_b}} d\phi_{ai}d\phi_{ij}d\phi_{jb}}_{0} = \int_{Y_{z_a,z_b}} \underbrace{d(d\phi_{ai}d\phi_{ij}d\phi_{jb})}_{0} \pm \int_{\partial Y_{z_a,z_b}} d\phi_{ai}d\phi_{ij}d\phi_{jb},$$ 
 where $Y_{z_a,z_b}$ is the configuration space of four points ($i,j,a$ and $b$) where $a$ and $b$ are fixed at $z_a$ and $z_b$ and the points labeled by $i$ and $j$ are free. The integral on the left hand side vanishes by degree reasons. The boundary terms on the right hand side vanish except on the following cases:
 
 \begin{itemize}
 \item The boundary stratum in which $a$ and $i$ are infinitely close,
 
 \item The boundary stratum in which $i$ and $j$ are infinitely close,
 
 \item The boundary stratum in which $j$ and $b$ are infinitely close.
 \end{itemize}
In each of these cases, the result is an integral of the form of integral \eqref{integral} (possibly with different signs), therefore it is zero.\\

If a graph $\Gamma$ contains an internal vertex with a tadpole, the form $f_1(\Gamma)$ includes a term of the form $d\phi$, where $\phi$ is the angle between the vertical direction and the frame at the corresponding point. However $\pi_\Gamma^{-1}(c)$ is a chain in which the frame of that point does not vary, therefore the integral $\int_{\pi_{\Gamma}^{-1}(c)}f_1(\Gamma)$ vanishes.
\end{proof}

As a consequence of Lemma \ref{lemma:natively twistable}, the canonical projections $Tw \CBr \to \CBr$ and $Tw \BVGraphs \to \BVGraphs$ admit right inverses. This defines a $\CBr-\BVGraphs$ bimodule structure on   $Tw \displaystyle\prod_n\BVKGra^\mu(\cdot, n)[-n]$. Elements of this bimodule are (sequences of) graphs with type I and type II vertices as before, but now there are two kinds of type I vertices. Using the same designations as in $\CBr$ we refer to the labeled type I vertices as external vertices and the indistinguishable type I vertices as internal vertices.

\begin{prop-def}
The $\CBr-\BVGraphs$ bimodule   $Tw \displaystyle\prod_n\BVKGra^\mu(\cdot, n)[-n]$ has a subbimodule that we call $\BVKGraphs$ that is spanned by the graphs with the following properties:
\begin{enumerate}[(1)]
\item There is at least one type I external vertex,
\item There are no 0-valent type I internal vertices
\item There are no 1-valent type I internal vertices with an outgoing edge,
\item There are no 2-valent type I internal vertices with one incoming and one outgoing edge (in particular there are no internal vertices with one tadpole and no other incident edges).
\end{enumerate}
\end{prop-def}

\begin{proof}
We must check that $\BVKGraphs$ is preserved by the differential, the left $\CBr$ and right $\BVGraphs$ actions. This is clear for the right $\BVGraphs$ action.

To check that $\BVKGraphs$ is closed under the action on $\CBr$ we start by considering the action of the generator $T_n^1$. Let $\Gamma_1,\dots,\Gamma_n$ be graphs in $\BVKGraphs$. The element $T_n^1(\Gamma_1,\dots,\Gamma_n)$ is determined by inserting  $\Gamma_2,\dots,\Gamma_n$ at the type II vertices of $\Gamma_1$, therefore every type I vertex in $T_n^1(\Gamma_1,\dots,\Gamma_n)$ can be identified as coming from one of the $\Gamma_i$. Since there are only incoming edges at type II vertices, the action of $T_n^1$ can increase or maintain the number of incoming edges at a type I vertex but it can only maintain the number of outgoing edges at every type I vertex, thus proving that properties \textit{(2)}, \textit{(3)} and \textit{(4)} are preserved. Property \textit{(1)} is clearly preserved.

The action of $T_n^j$ is given by insertions of the $\Gamma_i$ in the type II vertices on cyclic permutations of $\Gamma_1$, using the cyclic action of $\BVKGra$ described in section \ref{BVKGra}. Since the cyclic action preserves properties \textit{(1)}-\textit{(4)}, $\BVKGraphs$ is closed under the action of $T_n^j$.

The insertion of the empty graph $\mathbbm 1 \in \BVKGra(0,0)$ on some type II vertex of another graph has two possible outcomes. Either there is an incoming edge and the insertion of $\mathbbm 1$ at that vertex is $0$ or there were no incoming edges and the insertion of $\mathbbm 1$ forgets the vertex. In both cases properties  \textit{(1)}-\textit{(4)} are preserved, therefore $\BVKGraphs$ is closed under the action of $T_n^{j,j+1}$.

To show that $\BVKGraphs$ is closed under the action of $T_n'^j$, it is enough to check that summands of the Maurer-Cartan element by which $\prod \BVKGra^\mu(\cdot, n)[-n]$ was twisted (image of the generators of $\hoLie^{\text{bimod}}$) satisfy the following two properties:

\begin{enumerate}[(a)]
\item The only graph containing a $1$-valent type I internal vertex is the 2 vertex graph 
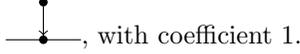
\begin{tikzpicture}[scale=1]

\node (i) at (1,0.5) [int] {};
\node (j) at (1,0) [int] {};

\draw [->] (i)--(j);
\draw  (0.5,0)--(1.5,0);
\end{tikzpicture}, with coefficient $1$.
\item There are no graphs with vertices like the ones in property \textit{(4)}.
\end{enumerate}

To verify these properties we recall that the map  $\hoLie^{\text{bimod}} \to \prod \BVKGra^\mu(\cdot, n)[-n]$ involves at some step the integration of differential forms over $\mathsf F \mathbb H_{\bullet, 0}$. Then, property (a) follows from degree reasons and property (b) has a proof similar to Proposition \ref{lemma:factors through BVGraphs}.

It remains to check that the differential preserves $\BVKGraphs$.

The differential is composed of the following pieces:
\begin{itemize}
\item The original splitting of type II vertices,
\item Insertion of \begin{tikzpicture}
\node (e1) at (0,0) [ext] {};
\node (e2) at (0.5,0) [int] {};

\draw [->] (e1)--(e2);
\end{tikzpicture}$+$
\begin{tikzpicture}
\node (e1) at (0,0) [ext] {};
\node (e2) at (0.5,0) [int] {};

\draw [->] (e2)--(e1);
\end{tikzpicture} at type I external vertices,

\item Insertion of \begin{tikzpicture}
\node (e1) at (0,0) [int] {};
\node (e2) at (0.5,0) [int] {};

\draw [->] (e1)--(e2);
\end{tikzpicture}
 at type I internal vertices,

\item Bracket with the image of the generators of $\hoLie^{\text{bimod}}$.
\end{itemize}

The first piece of the differential clearly preserves $\BVKGraphs$. Properties \textit{(1)} and \textit{(2)} are trivially preserved by all pieces of the differential. It remains to check properties \textit{(3)} and \textit{(4)}. The remaining pieces of the differential can produce vertices like \textit{(3)} and \textit{(4)}, so we must verify that these graphs cancel. There are 3 possibilities to obtain a vertex of the kind \textit{(3)} with the differential: 

Using the second piece of the differential on a graph $\Gamma \in \BVKGraphs$, at every external vertex we get a forbidden 1-valent vertex connecting to it, corresponding to inserting \begin{tikzpicture}
\node (e1) at (0,0) [ext] {};
\node (e2) at (0.5,0) [int] {};

\draw [->] (e2)--(e1);
\end{tikzpicture} and reconnecting all the originally incident edges to the external vertex.
Similarly, for every internal vertex of $\Gamma$, the second piece of the differential produces one 1-valent internal vertex with one outgoing edge connecting to it.

Due to property (a), the only ``problematic" graphs that may arise from the fourth piece of the differential are coming from bracket with \begin{tikzpicture}[scale=1]

\node (i) at (1,0.5) [int] {};
\node (j) at (1,0) [int] {};

\draw [->] (i)--(j);
\draw  (0.5,0)--(1.5,0);
\end{tikzpicture}. The bracket with this element gives
 \begin{tikzpicture}[scale=1]
\node (i) at (1,0.5) [int] {};
\node (j) at (1,0)  {$\Gamma$};


\node at (2,0.25) {$\pm\sum_i\Gamma \circ_{\overline i}$};

\draw [->] (i)--(j);

\node (i2) at (3.2,0.5) [int] {};
\node (j2) at (3.2,0) [int] {};

\draw [->] (i2)--(j2);
\draw  (2.7,0)--(3.7,0);
\end{tikzpicture}
 where on the first summand we connect the internal vertex to every possible (type I or II) vertex of $\Gamma$ and on the second summand the $\circ_i$ represents an insertion at the vertices of $\Gamma$ of type II. 
 In  \begin{tikzpicture}[scale=1]\node (i) at (1,0.5) [int] {};
\node (j) at (1,0)  {$\Gamma$};

\draw [->] (i)--(j);
\end{tikzpicture}, the edges connecting to type I vertices in $\Gamma$ are all canceled out with the second and third pieces of the differential as described above. The edges connecting to type II vertices are canceled by the terms in  \begin{tikzpicture}[scale=1]

\node at (2,0.25) {$\sum_i\Gamma \circ_{\overline i}$};

\node (i2) at (3.2,0.5) [int] {};
\node (j2) at (3.2,0) [int] {};

\draw [->] (i2)--(j2);
\draw  (2.7,0)--(3.7,0);
\end{tikzpicture} in which all the incident edges to ${\overline i}$ are reconnected to the type II vertex after the insertion.\\

To check that the differential preserves property \textit{(4)}, one can see that everytime an internal vertex having property \textit{(4)} is created due to type I internal or external vertex splitting, this is term is canceled by a splitting on the other adjacent vertex to the 2-valent vertex that was created. This also holds for splitting of vertices adjacent to type II vertices, but in that case the cancellation is done with a term coming from \begin{tikzpicture}[scale=1]

\node at (2,0.25) {$\sum_i\Gamma \circ_{\overline i}$};

\node (i2) at (3.2,0.5) [int] {};
\node (j2) at (3.2,0) [int] {};

\draw [->] (i2)--(j2);
\draw  (2.7,0)--(3.7,0);
\end{tikzpicture}.  Due to property (b), no more forbidden graphs are produced by the fourth piece of the differential.
\end{proof}

\begin{lemma}
$Chains(\mathbb H_{\bullet,0})$ is natively twistable.
\end{lemma}
\
The construction of the map $Chains(\mathbb H_{\bullet,0})\to Tw \ Chains(\mathbb H_{\bullet,0})$ is identical to Lemma \ref{lemma: chains natively twistable} and the compatibility with the left and right actions is immediate.\\

As a consequence, the bimodule morphism $Chains(\mathbb H_{\bullet,0}) \to \BVKGra$ factors through $Tw \displaystyle\prod_n\BVKGra^\mu(\cdot, n)[-n]$. The explicit formula is given by

\begin{equation}\label{map:chains to bvkgraphs}
c \in    Chains(\mathbb H)(n) \mapsto \sum_\Gamma \Gamma \int_{\pi_{\Gamma}^{-1}(c)}f_2(\Gamma),
\end{equation}
where $f_2(\Gamma)$ is the form associated to the graph $\Gamma$, as defined in Section \ref{subsec: chains to graphs} and if $\Gamma$ is a graph with $n$ external and $m$ internal type I vertices and $k$ type II vertices, $\pi^{-1}_{\Gamma}(c)$ is the chain in $Chains(\mathbb H_{n+m,k})$ in which the $m$ points corresponding to the internal vertices vary freely in the upper half plane while their frame is constantly pointing upwards.

\begin{proposition}
The bimodule morphism $Chains(\mathbb H_{\bullet,0}) \to \BVKGra$ factors through $\BVKGraphs$.
\end{proposition}

The proof is essentially the same as the one of Proposition \ref{lemma:factors through BVGraphs}.

\subsection{The Twist}

As a consequence of the previous section we have the following map of bimodules representing the last layer of the formality morphism:

\begin{equation*}
\begin{tikzcd}[column sep=0.5em]				
\CBr \arrow{d}& \aol  & \BVKGraphs \arrow{d} &\aor & \BVGraphs \arrow{d}\\
\End \Dpoly(\Rformal) &  \aol  &  \Hom(\Tpoly^{\otimes \bullet}(\Rformal),\Dpoly(\Rformal))      &  \aor   & \End \Tpoly(\Rformal).\\
\end{tikzcd}
\end{equation*}

As described in section \ref{subsection:idea}, the fibers of the vector bundles $\mathcal T_\text{poly}$ and $\mathcal D_\text{poly}$ are isomorphic to $\Tpoly(\Rformal)$ and $\Dpoly(\Rformal)$, therefore this induces the following map of bimodules:

\begin{equation*}
\begin{tikzcd}[column sep=0.5em]				
\CBr \arrow{d}& \aol  & \BVKGraphs \arrow{d} &\aor & \BVGraphs \arrow{d}\\
\End \Omega(\mathcal D_\text{poly}) &  \aol  &  \Hom(\Omega(\mathcal T_\text{poly})^{\otimes \bullet},\Omega(\mathcal D_\text{poly}))      &  \aor   & \End\Omega(\mathcal T_\text{poly}).\\
\end{tikzcd}
\end{equation*}

Since the $\CBr_\infty$ formality morphism from Theorem \ref{main theorem} is an extension of Kontsevich's $L_\infty$ formality morphism (see section \ref{subsection:extension}), its $L_\infty$ part satisfies properties P1)-P5) from  section 7 in \cite{Kontsevich}. In particular, property P4) implies that for $n\geq 2$, $U_n(B,\dots,B)=0$ and thus $B' = \sum_{n=1}^\infty \frac{1}{n!}U_n(B,\dots,B) = U_1(B) = B$, under the identification $\Omega^1(\mathcal T^1_\text{poly}) = \Omega^1(\mathcal D^1_\text{poly})$.

On the other hand, the bimodule $\BVKGraphs$ is obtained from a twist therefore it is natively twistable. 

Therefore, following the Appendix, we obtain a map of bimodules:

\begin{equation*}
\begin{tikzcd}[column sep=0.5em]				
\CBr \arrow{d}& \aol  & \BVKGraphs \arrow{d} &\aor & \BVGraphs \arrow{d}\\
\End \Omega(\mathcal D_\text{poly})^B &  \aol  &  \Hom\left(\left(\Omega(\mathcal T_\text{poly})^B\right)^{\otimes \bullet},\Omega(\mathcal D_\text{poly})^B)\right)      &  \aor   & \End\Omega(\mathcal T_\text{poly})^B,\\
\end{tikzcd}
\end{equation*}

where the superscript $B$ indicates that we are considering the twisted differential. For this twist it is important that $\BVKGraphs$ forbids $1$-valent internal vertices with an outgoing edge and $2$-valent internal vertices with one incoming and one outgoing edges, since the linear part of $B$ is not globally well-defined.

Composing with this map with bimodule maps $\CBr^{\text{bimod}}_\infty \to Chains(\mathbb H_{\bullet,0}) \to \BVKGraphs$, we obtain the desired global $\CBr_\infty$ quasi-isomorphism.

\appendix

\section{Twisting}
In this Appendix we give an overview on the theory of operadic twisting following \cite{Operadic Twisting} that we need for this paper and we define a notion of twisting of operadic bimodules, which is not more than an adaptation of the same theory. We advise the reader to read the third section of loc. cit. if they are not familiar with twisting of operads.

We make the standard assumptions used in the context of twisting with respect to Maurer-Cartan elements. Namely, all algebras $\mathfrak g$ (over $\hoLie$ or another operad) are equipped with complete decreasing filtrations 
$\mathfrak g = F_0 \mathfrak g \supset F_1 \mathfrak g \supset \dots$, such that the operations are compatible with the filtration and $\mathfrak g =  \displaystyle \varprojlim_i \mathfrak g/F_i\mathfrak g.$ These assumptions are made so that infinite sums (going deeper in the filtration) are allowed.

Let $\mathfrak g$ is a $\hoLie$ algebra, an element $\mu\in F_1 \mathfrak g$ of degree $2$ is said to be \textit{Maurer-Cartan} element of $\mathfrak g$ if it satisfies the Maurer-Cartan equation:
$$d\mu + \sum_{n=2}^\infty \frac{1}{n!}l_n(\mu,\dots,\mu) = 0,$$
where $l_n$ are the generating operations in $\hoLie$.

Given such a Maurer-Cartan element, one can construct the twisted $\hoLie$ algebra $\mathfrak g^\mu$, that is as a graded vector space just $\mathfrak g$, but with a changed (called twisted) differential, denoted by $d^\mu$, that is defined by $d^\mu(x) = dx + \sum_{n=1}^\infty  \frac{1}{n!}l_{n+1}(\mu,\dots,\mu,x)$, and new brackets given by $l^\mu_n(x_1,\dots,x_k) = \sum_{n=1}^\infty  \frac{1}{n!}l_{n+k}(\mu,\dots,\mu,x_1,\dots,x_k)$.

\subsection{Twisting of operads}
Let $\cP$ be an operad and let us assume the existence of an operad morphism $F\colon\hoLie\to \cP$. If $\mathfrak g$ is a $\cP$ algebra, thanks to $F$. Therefore it makes sense to talk about Maurer-Cartan elements of $\mathfrak g$. If $\mu$ be a Maurer-Cartan element of $\mathfrak g$, the twisted algebra $\mathfrak g^\mu$ is no longer necessarily a $\cP$ algebra. It is, however, an algebra over the operad $Tw \cP$, whose construction depends on the map $F$.

As an $\mathbb S$-module, we have $$Tw \cP(p) = \prod_{r\geq 0} \left(\cP(r+p) \otimes \mathbb K[-2r]\right)^{\mathbb S_r},$$
where $\mathbb S_r$ here is the subgroup of $\mathbb S_{r+p}$ fixing the last $p$ entries. The $r$ non-symmetric inputs should be thought as representing the insertion of $r$ Maurer-Cartan elements. The composition is defined using the original composition in $\cP$, but summing over suffles to ensure that it lands in the invariants over the action of $\mathbb S_{r_1+r_2}$. 

To describe the differential we need an auxiliary dg Lie algebra:
$$\mathcal L_\cP := \text{Conv}(\Lie\{1\}^{\vee}, \cP) = \prod_{n\geq 1} \cP(n)^{\mathbb S_n}[2-2n].$$
The Lie algebra $\mathcal L_\cP$ acts on $Tw \cP$, by composition on the non-symmetric inputs. $Tw \cP(1)$ acts on $Tw \cP$ by inner derivations.

There is an obvious degree zero map $\kappa\colon \mathcal L _\cP \to \mathcal Tw \cP(1)$.

The map $F$ induces a Maurer-Cartan element $\tilde{F}$, and the final differential is $d_{Tw} = d_\cP + d_{\tilde F} + d_{\kappa( \tilde F)}$, where the first piece is induced by the original differential in $\cP$, the second one comes from the $\mathcal \mathcal L_P$ action and the third one comes from the $Tw \cP(1)$ action.

The fact that this is a differential is essentially a consequence of the following Proposition:

\begin{proposition}\cite[Prop. 3.3]{Operadic Twisting}\label{lem: map to semidirect product}
The map 
\begin{align*}
&\mathcal L_\cP &\to &\mathcal L_\cP \ltimes Tw \cP(1)&\\
& v &\mapsto &v + \kappa(v)&
\end{align*}
 is a morphism of Lie algebras.
\end{proposition}

The action of $Tw \cP$ on $\mathfrak g^\mu$ is given by inserting Maurer-Cartan elements in the non-symmetric slots. Explicitly, let $p\in Tw \cP(n)$ and let $x_1,\dots, x_n\mathfrak g^{\mu}$. 
$$p(x_1,\dots,x_n) := \sum_{r=0}^\infty \frac 1 {r!}p_r(\mu,\dots,\mu, x_1,\dots,x_n),$$
where $p_r$ is the projection of $p$ in the factor $\left(\cP(r+n) \otimes \mathbb K[-2r]\right)^{\mathbb S_r}$.

There is a natural operad projection map $Tw \cP \to \cP$, sending $\prod_{r\geq 1} \left(\cP(r+n) \otimes \mathbb K[-2r]\right)^{\mathbb S_r}$ to zero. At the algebra level this tells us that not only twisted $\mathfrak g^{\mu}$ but the original $\mathfrak g$ are naturally $Tw \cP$ algebras.

On the other direction, an operad $\cP$  is said to be \textit{natively twistable} if there exists an operad morphism $\cP\to Tw \cP$ such that $\cP\to Tw \cP\to \cP$ is the identity. In this case, the twist of a $\cP$-algebra is still a $\cP$-algebra.

\begin{lemma}\cite[Lemma 16]{Note}  \label{lemma:natively twistable}
Let $\cP$ be an operad (with an implicit map $\hoLie\to \cP$. $Tw\cP$ is natively twistable.
\end{lemma}
Notice that $Tw\ Tw \cP(n) = \displaystyle\prod_{r_1,r_2\geq 0}\left(\left(\cP((n+r_1)+r_2) \otimes \mathbb K[-2r_1] \otimes \mathbb K[-2r_2]\right)^{\mathbb S_{r_1}}\right)^{\mathbb S_{r_2}}=$

$= \displaystyle\prod_{r_1,r_2\geq 0}\left(\cP(n+r_1+r_2) \otimes \mathbb K[-2(r_1+r_2)]\right)^{\mathbb S_{r_1}\times \mathbb S_{r_2}}= \displaystyle\prod_{r\geq 0}\prod_{r_1+r_2=r}\left(\cP(n+r) \otimes \mathbb K[-2r]\right)^{\mathbb S_{r_1}\times \mathbb S_{r_2}}$.

For $p\in Tw \cP(n)$, the map $Tw \cP \to Tw\ Tw \cP$ is defined by the inclusion of $p_r\in \left(\cP(n+r)  \otimes \mathbb K[-2r]\right)^{\mathbb S_{r}}$ in the factors of $Tw\ Tw \cP(n)$ in which $r_1+r_2 = r$ and zero if $r_1+r_2\ne r$.

\subsection{Twisting of bimodules}

Let $\mathfrak{g}$ and $\mathfrak{h}$ be $\hoLie$ algebras. Given an infinity morphism from $\mathfrak{g}$ to $\mathfrak{h} $, we define a \textit{Maurer-Cartan} element of this morphism to be a pair $(\mu,\mu')$, where  $\mu$ is a Maurer-Cartan element of $\mathfrak{g}$ and  $\mu'$ is  a Maurer-Cartan element of $\mathfrak{h}$ such that the $\hoLie$ morphism sends $\mu$ to $\mu'$ \footnote{Evidently for a fixed $\hoLie$ infinity morphism, $\mu$ determines a unique $\mu'$.}.

Let $\cP$ and $\cQ$ be (dg) operads and $\cM$ be a $\cP-\cQ$ operadic bimodule, that we assume to come with an implicit bimodule morphism $F\colon\hoLie^\text{bimod} \to \cM$.

\begin{equation*}
\begin{tikzcd}[column sep=0.5em]				
\hoLie\arrow{d}{F_\cP}& \aol  & \hoLie^\text{bimod}\arrow{d}{F} &\aor & \hoLie \arrow{d}{F_\cQ}\\
\cP & \aol  & \cM  &\aor & \cQ\\
\end{tikzcd}
\end{equation*}

Let $\mathfrak{g}$ be a $\cP$ algebra and let $\mathfrak{h}$ be a $\cQ$ algebra. Due to the map $F$, a morphism of bimodules

\begin{equation}\label{eq:non twisted action}
\begin{tikzcd}[column sep=0.5em]				
\cP \arrow{d}& \aol  & \cM \arrow{d} &\aor & \cQ \arrow{d}\\
\End \mathfrak{h} &  \aol  &  \Hom(\mathfrak{g}^{\otimes \bullet},\mathfrak{h})      &  \aor   & \End \mathfrak{g}\\
\end{tikzcd}
\end{equation}

 determines a $\hoLie$ infinity morphism from $\mathfrak{g}$ to $\mathfrak{h}$. We wish to construct a $Tw \cP-Tw \cQ$ bimodule $\cM$ such that for every  $(\mu,\mu')$, Maurer-Cartan element of this morphism , there is a natural map of bimodules 
 \begin{equation*}
\begin{tikzcd}[column sep=0.5em]				
Tw\cP \arrow{d}& \aol  & Tw\cM \arrow{d} &\aor & Tw\cQ \arrow{d}\\
\End \mathfrak{h}^{\mu'} &  \aol  &  \Hom({\mathfrak{g}^{\mu}}^{\otimes \bullet},\mathfrak{h}^{\mu'})      &  \aor   & \End \mathfrak{g}^{\mu}\\
\end{tikzcd}
\end{equation*}

We start by giving the description of $Tw \cM$ as an $\mathbb S$-module.

\begin{defi}
The $Tw \cP-Tw \cQ$ bimodule  $Tw \cM$ is the space
$$Tw \cM(n) = \prod_{r\geq 0} \left(\cM(r+n) \otimes \mathbb K[-2r]\right)^{\mathbb S_r},$$
with differential $d_{Tw}$, where $\mathbb S_r$ here is the subgroup of $\mathbb S_{r+n}$ fixing the last $n$ entries.
\end{defi}

We need now to clarify the left and right actions, as well as the differential.

Let $m\in Tw \cM(n)= \prod_{r\geq 0} \left(\cM(p+r) \otimes \mathbb K[-2r]\right)^{\mathbb S_r}$. We denote by $m_r$ it's projection in $\left(\cM(p+r) \otimes \mathbb K[-2r]\right)^{\mathbb S_r}$ and for $p\in Tw P$, $q\in Tw \cQ$ we use a similar notation $p_r, q_r$. 

The right $Tw \cQ$ action on $\cM$ is defined in the following way:
Let $m\in Tw \cM(n)$ and $q\in Tw \cQ(l)$. 
$$(m\circ_i q)_r := \sum_{p=0}^r \sum_{\sigma\in Sh_{p,r-p}} \gamma_{i,\sigma}(m_p,q_{r-p}),$$ 
where $Sh_{p,r-p}\subset \mathbb S_r$ are the $(p,r-p)$ shuffles $\gamma_{i,\sigma}$ is the composition given by the following tree

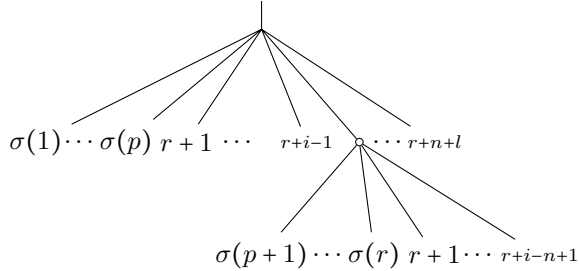
\begin{tikzpicture}

\node (star) at (0,2)  {};
\node (top) at (0,1.5)  {};

\node (e1) at (-3,0)  {$\sigma(1)$};
\node (e2) at (-2.4,0)  {$\dots$};
\node (e3) at (-1.8,0) {$\sigma(p)$};
\node (e4) at (-1,0)  {$r+1$};
\node (e5) at (-0.3,0)  {$\dots$};
\node (e6) at (0.6,0)  {$\scriptstyle{r+i-1}$};
\node (i) at (1.3,0) [ext] {};
\node (e7) at (1.7,0) {$\dots$};
\node (e8) at (2.3,0) {$\scriptstyle{r+n+l}$};

\node (l1) at (0,-1.5) {$\sigma(p+1)$};
\node (l2) at (0.9,-1.5) {$\dots$};
\node (l3) at (1.5,-1.5) {$\sigma(r)$};
\node (l4) at (2.3,-1.5) {$r+1$};
\node (l5) at (2.9,-1.5) {$\dots$};
\node (l6) at (3.7,-1.5) {$\scriptstyle{r+i-n+1}$};

\draw [-] (0,1.5)--(e1) (0,1.5)--(e3) (0,1.5)--(e4) (0,1.5)--(e6) (0,1.5)--(i) (0,1.5)--(e8);

\draw (0,1.5)--(star);

\draw (i) -- (l1)  (i) -- (l3)  (i) -- (l4)  (i) -- (l6);
\end{tikzpicture}

 We write $d_{Tw}= d_{\cM} + d_R + d_L$, where $d_\cM$ is the differential induced by the differential in $\cM$. 

The Lie Algebra $\mathcal L_\cQ$ acts on $(Tw \cM,d_\cM)$ by operadic derivations. The proof of this is the same as \cite[Proposition 3.2]{Operadic Twisting}.

The Lie Algebra $Tw \cQ(1)$ acts on the right on $Tw \cM$ by 
$$m \cdot q= \sum_{i=1}^n m\circ_i q,$$ 
where $m\in Tw \cM(n)$ and $q\in Tw \cQ$.

Multiplying by a minus sign, the previous right action becomes a left action, thus inducing  a dg Lie algebra action $\mathcal L_\cQ \ltimes Tw \cQ(1) \aol (Tw \cM, d_\cM)$. 

The map $F_\cQ\colon \hoLie \to \cQ$ gives us a Maurer-Cartan element in $\mathcal L_\cQ$. Due to Lemma \ref{lem: map to semidirect product} we can twist $(Tw \cM,d_\cM)$ with respect to this Maurer-Cartan element, giving us the module $(Tw \cM, d_\cM+d_R)$. 

There is an obvious left $\cP$ action on $(Tw \cM, d_\cM)$, using the original $\cP$ action on $\cM$. It is easy to see that $\cP$ also acts on $(Tw \cM, d_\cM + d_R)$. 

Indeed, the equation of compatibility with the differential
$$(d_\cM + d_R)(p\circ_i m) = d_\cP p \circ_i m +(-1)^{|p|} p \circ_i (d_\cM + d_R)m $$
is equivalent to $d_R(p\circ_i m) = (-1)^{|p|}p\circ_i d_Rm$, and the associativity axiom involving the left and right actions of an operadic bimodule, together with the fact that $d_R$ uses right compositions ensures that this equality holds for all $p\in Tw \cP$ and $m\in Tw \cM$.

The map $F\colon \hoLie^\text{bimod}\to \cM$ gives us a Maurer-Cartan element in $\prod_r \Hom_{\mathbb S_r} (\mathbb K[2r],\cM(r)) =\prod_r (\cM(r)\otimes \mathbb K[-2r])^{\mathbb S_r} = Tw \cM(0)$. Twisting with respect to this Maurer-Cartan element we obtain a left action of $Tw \cP$ on $(Tw\cM, (d_\cM+d_R) + d_L)$. 

Using a similar argument of compatibility with the differential, we see that $Tw \cQ$ acts on the right on $(Tw \cM, d_\cM + d_R + d_L)=(Tw \cM, d_{Tw})$. The associativity of the left $Tw \cP$ and right $Tw \cQ$ actions is clear and so we finished the construction of the bimodule $Tw \cM$.

\subsubsection{The action on $\Hom({\mathfrak g^\mu}^{\otimes \bullet}, \mathfrak h^{\mu'})$}
As described in the beginning of the section, we wish now to construct a map of bimodules
 \begin{equation}\label{eq:twisted action}
\begin{tikzcd}[column sep=0.5em]				
Tw\cP \arrow{d}& \aol  & Tw\cM \arrow{d} &\aor & Tw\cQ \arrow{d}\\
\End \mathfrak{h}^{\mu'} &  \aol  &  \Hom({\mathfrak{g}^{\mu}}^{\otimes \bullet},\mathfrak{h}^{\mu'})      &  \aor   & \End \mathfrak{g}^{\mu}\\
\end{tikzcd}
\end{equation}

The two outer maps are the maps induced by the usual twisting of operads. For the main map, informally we do the usual procedure of inserting the Maurer-Cartan element on the non-symmetric slots. Formally, if $m\in Tw\cM(n)$,

$$
m(x_1,\dots, x_n) = \sum_{r=0}^{\infty} \frac{1}{r!}m_r(\mu,\dots,\mu, x_1, \dots,x_n),\quad x_i\in \mathfrak g,$$
where we identify an element of $\cM$ (resp. $Tw \cM$) with its image in $\Hom(\mathfrak g, \mathfrak h)$ (resp. $\Hom(\mathfrak{g}^{\mu},\mathfrak{h}^{\mu'})$).

The only thing that remains to be checked is the commutativity of the left and right squares, as well as the compatibility with the differential of the central vertical map. Let as call $l_r^\cP$ the image of the $r$-ary generator of $\hoLie$ in $\cP$, and we define similarly $l_r^\cQ$ and $l_r^\cM$.

Due to the original bimodule morphism \eqref{eq:non twisted action}, the right square is trivially commutative and the commutativity of the left square is a simple consequence \eqref{eq:non twisted action} together with the hypothesis $\displaystyle\sum_r \frac{1}{r!} l^\cM_r(\mu,\dots,\mu)=\mu'$. Also, thanks to this equation, when we evaluate $d_Lm$ in $\Hom(g^\mu,h^{\mu'})$, $mu'$ will replace the Maurer-Cartan element of $Tw \cM$. 

We wish to show that for $m\in Tw \cM (n)$ and $x_1,\dots,x_n\in \mathfrak g^\mu$, 
$$d^{\mu'}(m(x_1,\dots,x_n)) = (d_\cM m + d_L m + d_R m)(x_1,\dots,x_n) +  \sum_{i=1}^n (-1)^{|m|+|x_1| + \dots + |x_{i-1}|} m(x_1,\dots,d^{\mu}x_i,\dots,x_n).$$   

Keep in mind in the following computations that $m_r$ has $r$ non-symmetric inputs and $n$ symmetric inputs, whereas $l^{\cP/\cQ}_r$ will be of arity $r$ but will have $r-1$ non-symmetric inputs.  Expanding the right hand side we get

$$\sum_{r\geq 0}\frac{1}{r!}d_\cM  m_r(\mu,\dots,\mu,x_1,\dots,x_n)+ \sum_{k\geq 2, r\geq 0}\frac{1}{(k-1)!r!} l^\cP_k(\mu',\dots,\mu', m_r(\mu,\dots,\mu,x_1,\dots, x_n))+ $$

$$-\sum_{r\geq 0, k\geq 2} \frac{r}{r!k!} m_r(l_k^{\cQ}(\mu,\dots,\mu),\dots,\mu,x_1,\dots,x_n) +$$

$$ -\sum_{r\geq 0,k\geq 2} \frac{(-1)^{|x_1|+\dots+|x_{i-1}|}}{r!(k-1)!} m_r(\mu,\dots,\mu,x_1,\dots,l^\cQ_k(\mu,\dots,\mu, x_i),\dots,x_n)+$$

$$-\sum_{i=1}^n \sum_{r \geq 0} \frac{(-1)^{|x_1|+\dots+|x_{i-1}|}}{r!} m_r(\mu,\dots,\mu, x_1, \dots, dx_i,\dots,x_n)+$$

$$\sum_{i=1}^n \sum_{r \geq 0, k\geq 2} \frac{(-1)^{|x_1|+\dots+|x_{i-1}|}}{r!(k-1)!} m_r(\mu,\dots,\mu, x_1, \dots, l^\cQ_k(\mu,\dots,\mu,x_i),\dots,x_n).$$

Using the Maurer-Cartan equation, the third summand simplifies to $$\sum_{r\geq 0} \frac{r}{r!} m_r(d\mu,\mu,\dots,\mu,x_1,\dots,x_n),$$
therefore, the first, third and fifth summands add up to $d(m(x_1,\dots,x_n))$, while the fourth and sixth summands cancel out, leaving us precisely with $d^{\mu'}(m(x_1,\dots,x_n))$.

\end{document}